\documentclass[11pt]{article}
\usepackage{glauber}
\usepackage{physics}
\usepackage{mathtools}

\usepackage{rgd}
\usetikzlibrary{arrows.meta,positioning}
\newlength{\capht}
\allowdisplaybreaks

\begin{document}

\title{Sampling from the Sherrington--Kirkpatrick Model up to Criticality}
\author{
    Hang Du\thanks{MIT. \texttt{hangdu@mit.edu}.} \and
    Brice Huang\thanks{Stanford University. \texttt{bmhuang@stanford.edu}.} \and
    Sidhanth Mohanty\thanks{Northwestern University. \texttt{sidhanth.mohanty@northwestern.edu}.} \and
    Amit Rajaraman\thanks{MIT. \texttt{amit\_r@mit.edu}.}
}
\date{\today}
\maketitle
\begin{abstract} 
    We prove that the simulated annealing algorithm efficiently samples from the Sherrington--Kirkpatrick model at inverse temperature $\beta < 1$ to an arbitrary inverse polynomial accuracy in total variation.
\end{abstract}


\thispagestyle{empty}
\setcounter{page}{0}
\newpage
\tableofcontents
\thispagestyle{empty}
\setcounter{page}{0}
\newpage



\section{Introduction}

We study the problem of sampling from the Sherrington--Kirkpatrick (SK) model. This is an Ising model $\mu_{\beta\bW}$ with the random interaction matrix $\beta\bW$, for $\bW \sim \GOE(N)$, that is, $\bW_{ij} = \bW_{ji} \sim \calN\left( 0 , \frac{1}{N} \right)$. More explicitly, the density of this distribution is given by
\begin{equation}
    \label{eq:def-sk}
    \mu_{\beta\bW}(\sigma) \propto \exp\left( \frac{\beta}{2} \left\langle \sigma, \bW\sigma\right\rangle \right) \mper
\end{equation}
There is a vast literature dedicated to the study of Ising models, due to their connections to physics, optimization, statistical inference, among others; see \cite{WJ08,MM09,FV17} and references therein.
The SK model in particular has provided a rich testbed for techniques, and has been the subject of much work for several decades. 
Some landmarks in the study of the SK model are the Parisi formula predicting its free energy \cite{Par79,Par80a,Par80b}, foundational works towards rigorous study \cite{GT02,Gue03,ASS03}, Talagrand's proof of the Parisi formula \cite{Tal06}, and Panchenko's resolution of the ultrametricity conjecture \cite{Pan13}; we refer the reader to \cite{Tal10,Tal11,Pan12} for more details.

Of special interest has been the task of \emph{sampling} from the SK model in the high-temperature regime $0 \le \beta < 1$. It is believed that the canonical Glauber dynamics Markov chain mixes rapidly at this temperature \cite{SZ81,SZ82,MPV87}.

Proving this has been a major open problem in the past decade, as outlined in \cite{Sim20} and \cite[Problem 15]{BDLNR26}.
We defer a detailed discussion of prior results to \Cref{subsec:related-work}.

While rapid mixing of the fixed-temperature Glauber dynamics remains open, in this paper, we study \emph{simulated annealing}, a simple modification of the standard Markov chain Monte Carlo paradigm that is frequently used in practice across a variety of areas, including VLSI design, combinatorial optimization, and scheduling among others; see \cite{KGV83,SS85,LA87,WLL88,JAMS89,YRN94,KAJ94,HJJ03}.

The simulated annealing algorithm is a time-dependent Markov chain where each step is from the Glauber dynamics Markov chain, but the inverse temperature starts at $0$ and is slowly incremented to $\beta$ over the course of running the Markov chain.
The intuitive picture a reader should keep in mind is that while a Markov chain to sample from a target distribution $\mu_{\beta \bW}$ may not mix from a generic initialization, it would nevertheless sample from $\mu_{\beta \bW}$ when initialized at $\mu_{(\beta-\delta)\bW}$.

Our main result is the following.

\begin{theorem}
    \label{th:main}
    Let $0 \le \beta < 1$ and let $a > 0$ be a (large) constant. With probability at least $1 - N^{-a}$ over $\bW \sim \GOE(N)$, the output distribution $\wh{\mu}$ of the simulated annealing algorithm run for $O_{a,\beta}(N \log N)$ steps of the Glauber dynamics per iteration, with inverse temperature gap $c/N$, run on the SK model $\mu_{\beta\bW}$, is such that $\dtv{\wh{\mu}}{\mu_{\beta\bW}} \le N^{-a}$.
\end{theorem}

This is the first algorithm sampling from the SK model for all $\beta < 1$ up to an arbitrary inverse polynomial accuracy---the previous best results \cite{Cel24,LSS26a} only attain a total variation distance that is slightly $o(1)$. 
Another contribution is that our algorithm of choice is based on a Markov chain, which has consequences for the spiked Wigner model. In this statistical inference task, for some $\bx \sim \{\pm 1\}^N$, we are tasked with inferring $\bx$ from $\bW + \frac{\lambda}{N} \bx\bx^\top$ for $\bW \sim \GOE(N)$. In particular, using the techniques of \cite{LMRRW24,RW26}, our result makes progress towards establishing that Markov chains succeed at non-trivial recovery for all $\lambda > 1$, matching the so-called BBP phase transition \cite{BBP05} for when non-trivial recovery is possible. Completing this proof would require a version of the above statement not just for the zero-field model, but even for non-zero external fields aligned with the all-ones direction.

We prove the above by establishing a \emph{weak \Poincare inequality} \cite{Aid98,RW01,Mat06,HMRW24} for the SK model. A distribution $\pi$ on $\{\pm 1\}^N$ is said to satisfy a $(\rho,\eps)$-weak \Poincare inequality if for all functions $f$,
\begin{equation}
    \label{eq:wpi}
    \calE(f,f) \ge \frac{1}{\rho} \left( \Var_{\pi}[f] - \eps \|f\|_{\infty}^2 \right) \mcom 
\end{equation}
where $\calE(f,f)$ is the standard Dirichlet form.
When $\eps = 0$, this devolves into the usual \Poincare inequality, equivalent to a spectral gap for the Glauber dynamics.
For nonzero $\eps$, this yields rapid mixing (down to total variation error $\approx \sqrt{\eps}$) from a warm start.
\cite{HMRW24} establishes that if all the $\{\mu_{\alpha\bW}\}_{0 \le \alpha \le \beta}$ satisfy $(\rho,\eps)$ weak \Poincare inequalities with $\eps$ a sufficiently small subconstant quantity, then this yields guarantees for simulated annealing.
We prove \Cref{th:main} by establishing a high-probability weak \Poincare inequality for each $\mu_{\beta\bW}$, and taking a union bound over the inverse temperatures encountered along the trajectory of simulated annealing.

\begin{theorem}
    \label{th:wpi}
    Let $0 < \beta < 1$ and $a > 0$ be constants. With probability at least $1 - N^{-a}$ over $\bW \sim \GOE(N)$, the SK model $\mu_{\beta\bW}$ satisfies an $(O(N) , N^{-a})$-weak \Poincare inequality.
\end{theorem}

To establish this weak \Poincare inequality, we use the \emph{localization schemes} framework of Chen--Eldan \cite{CE22}.
We elaborate more on the details of this framework and how we use it in \Cref{subsec:tech-overview}; using techniques from prior works, this reduces the proof of \Cref{th:main} to establishing high-probability covariance bounds for a suitably modified SK model with a random external field---this is the primary technical contribution of our work.
A special case of our analysis is a high-probability covariance bound for the zero-field SK model.

\begin{restatable}{theorem}{zerofieldcovbound}
    \label{th:zero-field-cov-bound}
    Let $0 < \beta < 1$ and $a > 0$. Then, for some bounded constant $L = L_{\beta,a}$, with probability at least $1 - N^{-a}$ over $\bW \sim \GOE(N)$, the SK model \eqref{eq:def-sk} is such that $\Cov\left( \mu_{\beta\bW} \right) \psdle L\cdot\Id$.
\end{restatable}


\subsection{Technical overview}
\label{subsec:tech-overview}

In this section, we provide an overview of how we reduce \Cref{th:wpi} to establishing appropriate covariance bounds using localization schemes.
Localization schemes are a principled approach to designing measure decompositions, providing a measure-valued martingale $(\mu_{t})_{0 \le t \le T}$ with $\mu_0$ being the Gibbs measure.
For this decomposition to be helpful in proving a weak \Poincare inequality for the Gibbs measure, we will desire the following two properties:
\begin{enumerate}
    \item First, the (random) endpoint measure $\mu_T$ is ``nice'' in some sense.
    In our case, this will mean proving that $\mu_T$ satisfies a \Poincare inequality with high probability, say at least $1-\eps$.
    \item Second, we may establish some pathwise control that will allow us to lift this niceness back up, to proving a weak \Poincare inequality for the base measure $\mu_0$.
    In our case, this will mean proving that with high probability, say at least $1-\eps$, over the localization path, the covariances of all the $\mu_t$ are spectrally bounded by some constant (independent of $N$).
\end{enumerate}
Given this, we can execute the following string of inequalities, for $\calE$ being the event that the endpoint measure satisfies a \Poincare inequality, so $\Pr[\calE] \ge 1 - \eps$:
\begin{align*}
    \calE_{\mu}(f,f) &\ge \E \calE_{\mu_T}(f,f) \\
        &\gtrsim \E \Var_{\mu_{T}}[f] \bone_{\calE} \\
        &\ge \E \Var_{\mu_{T}}[f] - 4\eps \|f\|^2_{\infty} \\
        &\gtrsim \Var_{\mu}[f] - C\eps \|f\|^2_\infty \mper
\end{align*}
In the above $\gtrsim$ hides $\mathrm{poly}(N)$ factors.
The non-trivial inequalities here are the second and fourth: the second uses the high-probability \Poincare inequality for the endpoint measure, and the fourth uses the pathwise control.
The latter uses It\^{o}'s formula to compute the time-derivative of the expected variance, and controls this with the aforementioned ``pathwise control'' of the $(\mu_t)$ measures; we refer the interested reader to \cite[Section 4]{CE22} and \cite[Section 6]{HMRW24} for more details.

Our localization scheme of choice will be \emph{stochastic localization} \cite{Eld13}. For a measure $\mu_0 = \mu$ on $\{ \pm 1 \}^N$, stochastic localization is defined by the following measure-valued martingale:
\begin{equation}
    \label{eq:sl-sde}
    \dif \mu_t(x) = \left\langle x - \bary(\mu_t) , \dif B_t \right\rangle \cdot \mu_t(x) \mper
\end{equation}
\cite{EM22} provides an alternate perspective on this, which will be far more helpful for us.
Draw $\bsigma \sim \mu$ and a Brownian motion $(B_t)_{t \ge 0}$. Then, the process $(\mu_t)_{t \ge 0}$ is equivalently given by the measure, which at time $t$, is the posterior on $\bsigma$ conditioned on $\bz_t = t\bsigma + B_t$.
At time $0$, $\bz_t$ is uninformative so this is simply the Gibbs measure.
On the other hand, at infinite time, $\frac{1}{t} \bz_t$ is simply equal to $\bsigma$, and the posterior is then just the Dirac delta on this (random) point $\bx$ which is marginally distributed according to $\mu$. 
This posterior measure has an easy explicit formula:
\begin{equation}
    \label{eq:sl-trajectory}
    \mu_t(y) \propto \exp\left( \frac{\beta}{2} y^\top \bW y + \langle t \bsigma + B_{t} , y\rangle \right) \mper
\end{equation}
We now wish to show that $\Cov(\mu_t)$ is bounded with high probability, and that $\mu_{T}$ satisfies a \Poincare inequality with high probability.
However, as stated, this is rather intractable: the random variable $\bsigma$ depends on the disorder $\bW$ in some complicated way, which makes the analysis of $\mu_{t}$ very difficult.

To get around this issue, we shall shift to the so-called \emph{planted} model, first used by Achlioptas--Coja-Oghlan \cite{AC08} in random CSPs. 
This technique is now standard in the study of stochastic localization on the SK model \cite{EMS22,EMS23,HMP24,HMRW24}.
The planted model $\pi$ is defined by the following random Hamiltonian: draw $\bx \sim \mathrm{Unif}\{\pm 1\}^N$, and $\bW \sim \GOE(N)$, and set
\[ \pi(y) \propto \exp\left( \frac{\beta}{2} \left( y^\top \bW y + \frac{\beta}{N} \langle \bx,y \rangle^2 \right) \right) \mper \]
The planted model is helpful for two reasons.
The first, perhaps unsurprisingly, is that we may argue about it instead of the original SK model; they are mutually contiguous, in that high-probability events for either Hamiltonian transfer to high-probability events for the other, in a manner made precise in \Cref{lem:planted-contiguity}.
For example, we may use this statement with the event being ``the model satisfies a weak \Poincare inequality.''

The second, which is more interesting, is that the trajectory of stochastic localization is significantly more tractable in this model. In fact, the posterior on $\bx$ conditioned on $\bM \defeq \bW + \frac{\beta}{N} \bx\bx^\top$ is precisely the distribution $\pi$. 
This allows us to swap out the sample $\bsigma$ used in \eqref{eq:sl-trajectory} for the hidden ``plant'' $\bx$! That is, we are interested in models of the form
\[ \pi_{t}(y) \propto \exp\left( \frac{\beta}{2} y^\top \left( \bW + \beta \frac{\bx\bx^\top}{N} \right) y + \left\langle t\bx + \bg\sqrt{t} , y \right\rangle \right) \mcom \]
for $\bg \sim \calN(0,\Id)$. We shall prove a high-probability covariance bound for each fixed $t$ (not depending on the disorder), then use a union bound and a simple perturbation argument to translate this to a high-probability covariance bound for all $t$.
The problem then reduces to establishing high-probability covariance bounds for these models.

\begin{restatable}{theorem}{thcovbounds}
    \label{th:cov-bounds}
    Let $0 < \beta < 1$ and let $a > 0$ be a (large) constant. Fix $t > 0$ bounded, possibly decaying with $N$. For $\bg \sim \calN(0,\Id)$ and $\bx \sim \mathrm{Unif}\{\pm 1\}^N$, consider the probability measure $\mu_{t}$ on $\{\pm 1\}^N$ with
    \[ \mu_t(y) \propto \exp\left( \frac{\beta}{2} y^\top \left( \bW + \beta \frac{\bx\bx^\top}{N} \right) y + \left\langle t\bx + \bg\sqrt{t} , y \right\rangle \right) \mper \]
    Then, there is some constant $L = L_{\beta,a}$ such that with probability $1 - O(N^{-a})$, $\|\Cov(\mu_t)\| \le L$. 
\end{restatable}

The technical bulk of this paper will go towards proving the above, and we provide an overview of some ideas involved in \Cref{sec:warmup}.

At a high level, we prove the above as follows. We consider a prediction $\bK^{-1}$ for the covariance matrix based on the Thouless--Anderson--Palmer (TAP) approximation to the free energy, the ``inverse TAP Hessian''. This matrix has a term linear in the disorder, and a handful of terms involving the mean of the measure. In an entry of the $\bK \cdot \Cov$ that involves an element $\bW_{ij}$ of the disorder and some cumulant of the Gibbs measure, we may then decouple the cumulant from the disorder using the cavity method, at the cost of introducing a small (random) external field. We can now Taylor expand any given cumulant in terms of this small random external field; this introduces terms involving zero-field SK models (on $N-1$ spins) and Gaussians independent of the disorder defining these models. Gaussian concentration allows us to control the error in these approximations, using only the crude bounds that are yielded by overlap concentration, which are later refined using other means.
Finally, showing that this matrix $\bK$ is bounded away from $0$ requires other machinery, which we may establish using techniques from previous work \cite{Cel24}.

We further believe that many of our techniques to prove covariance bounds should extend to the entire replica-symmetric regime \cite{AT78} given sufficiently good overlap concentration, which might follow from recent results \cite{Lop26,KN26}. We have made no attempt to do this extension.

We must also show that at sufficiently large $T$, $\pi_{T}$ satisfies a \Poincare inequality with high probability. This near-immediately follows from the techniques of the previous work of Bandeira--El Alaoui--R\"{o}dder \cite{BER26}.

\begin{restatable}[\cite{BER26}]{lemma}{largetpoinc}
    \label{lem:large-t-poinc}
    Let $\mu_{T}$ be the measure obtained by running stochastic localization on $\mu_{\beta\bW + \frac{\beta^2}{N} \bx\bx^\top}$ for bounded time $T > 0$. Then, for a sufficiently large (constant) $T$, with probability $1 - e^{-cN}$, $\mu_{T}$ has \Poincare constant $O(N)$.
\end{restatable}

\subsection{Related work}
\label{subsec:related-work}

\parhead{Sampling from spin glasses.} It was shown in a series of works \cite{BB19,EKZ22,CE22} that the Glauber dynamics satisfies a \Poincare inequality for $\beta < 1/4$, which was improved on by Anari--Koehler--Vuong \cite{AKV24} to show that it succeeds up to $\beta \approx 0.295$; both of these proofs use stochastic localization.
More recently, the ``Bochner method'' was used by Wang \cite{Wan26} and in a follow-up work by Boban--Li--Oveis Gharan \cite{BLO26} to establish fast mixing up to $\beta = \frac{1}{2} + \eps$ (for some small constant $\eps$).

An independent line of work pioneered by El Alaoui, Montanari, and Sellke \cite{EMS22} studies efficient sampling via \emph{algorithmic stochastic localization}, based on simulating the stochastic differential equation obtained from the time-reversal of \eqref{eq:sl-sde}.
It was first shown in \cite{EMS22,Cel24} that this algorithm achieves (normalized) Wasserstein distance $o(1)$ to the Gibbs measure.
The work of Huang, Montanari, and Pham \cite{HMP24} illustrates how these ideas can be used to design algorithms achieving a total variation distance guarantee of $o(1)$ using this technique in the related spherical $p$-spin model.
Recently, works of Davies, Lee, Sandhu, and Shi \cite{DLSS26a,LSS26a} extended these ideas to design an algorithm that achieves total variation distance $o(1)$ to the Gibbs measure for the SK model.

The same set of authors has an accompanying line of work \cite{DLSS26b}, which builds on these ideas to establish a $\left( O(N) , o(1) \right)$-weak \Poincare inequality for the Gibbs measure when $\beta < 1/2$. However, their quantitative weakness parameter is not small enough to establish success of simulated annealing, and they must use the Langevin diffusion on another distribution to generate a sufficiently warm start.

Beyond the SK model, the ``$2$-spin'' model, the problem of sampling from the more general class of $p$-spin models has been studied.
These are defined by coefficients $(\beta_p)_{0 \le p \le p_*}$ for some constant $p_*$, and have density
\[ \mu(y) \propto \exp\left( \sum_{p \ge 0} \frac{\beta_p}{N^{(p-1)/2}} \langle G^{(p)} , y^{\otimes p} \rangle \right) \mcom  \]
for $G^{(p)}_{i_1,\ldots,i_p} \sim \calN(0,1)$ iid.
It is predicted that these models exhibit a richer phase diagram, with a true \Poincare inequality for $\beta$ up to a uniqueness threshold $\beta_\uniq$ \cite{BBM96,AJ24,AJKPV24}, that Glauber dynamics initialized uniformly at random succeed up to a shattering threshold $\beta_\shatter$ \cite{CHS93}, and that Markov chain-based methods fail at sampling past $\beta_\shatter$ \cite{CHS93,AJ24,EMS25,El25}.

The spherical setting, where the state space is $\sqrt{N} \cdot \bbS^{N-1}$ instead of $\{\pm 1\}^N$, is generally more tractable.
The corresponding $2$-spin model does not suffer any phase transition in sampling, with efficient sampling possible at any $\beta$ \cite{GLLR21}, but higher spin models are not fully understood.
Of note are \cite{HMP24}, which uses algorithmic stochastic localization, and \cite{HMRW24}, which analyzes simulated annealing using weak \Poincare inequalities, to sample up to a threshold $\beta_\SL$ that, for pure $p$-spin models, is within a universal constant of the predicted sharp ``shattering'' threshold $\beta_\shatter$.
While algorithmic stochastic localization provably cannot do better, it seems plausible that simulated annealing reaches $\beta_\shatter$.

\parhead{Covariance bounds.} Proving covariance bounds for the SK Gibbs measure, as in \Cref{th:cov-bounds}, has also been studied, mostly towards the goal of sampling.
Talagrand \cite{Tal11} showed that when $\beta < 1$, $\E \|\Cov(\mu_{\beta\bW})\|_{\mathrm{op}} = O\left( \log^8 N \right)$, and explicitly posed as an open question whether $\E \|\Cov(\mu_{\beta\bW})\|_{\mathrm{op}} = O(1)$ when $\beta < 1$.
A more precise question one may ask is whether the ``inverse TAP Hessian'' $\bK^{-1} = \left( (1+\beta^2) \Id - \beta \bW \right)^{-1}$ is an accurate predictor of the covariance of the zero-field SK model.
It is believed that with high probability, $\opnorm{\bK \Cov - \Id} = O\left( \frac{1}{\sqrt{N}} \right)$.
In fact, a suitable modification of this belief is expected to hold not just for the high-temperature zero-field model, but for replica-symmetric models with a sufficiently large random external field, below the so-called de Almeida--Thouless line \cite{AT78}.


All the results mentioned in the previous paragraph on sampling from the SK model with Markov chains, which establish \Poincare inequalities, also immediately yield (high-probability) covariance bounds, but the best such result only holds for $\beta < \frac{1}{2} + \eps$. El Alaoui--Gaitonde \cite{EG24}, building on estimates in \cite{ABSY21}, successfully resolved Talagrand's conjecture and showed that $\E \left\| \bK \Cov - \Id \right\|_{F}^2 = O(1)$ for all $\beta < 1$. \cite{BSXY24,BXY23} were able to establish high-probability covariance bounds in a setting where there is an external field $h$---their regime of $(\beta,h)$ is a large subset of the predicted replica-symmetric regime \cite{Tal10,Tal11}.

\parhead{Our techniques.}
We now discuss the works that inspired our techniques.
The local-to-global principle reducing a weak \Poincare inequality to high-probability covariance bounds was inspired by the localization schemes based approach to proving \Poincare inequalities for Markov chain \cite{ALOV24,EKZ22,CE22,LMRW24,AKV24}.
The use of these methods to establish weak \Poincare inequalities appears in the works of \cite{AEGP23,HMRW24,DLSS26b}.

The specific localization scheme of choice---stochastic localization---was introduced by Eldan \cite{Eld13}, and subsequently developed extensively in work on the KLS conjecture \cite{LV24,Che21b,KL22,SZ26,BKL26}.

Our proof of the covariance bounds relies on proving that the covariance is close to a prediction for $\bK^{-1}$ that arises from the TAP equations for the free energy of the SK model, developed by Thouless--Anderson--Palmer \cite{TAP77}.
To prove that $\bK^{-1}$ is bounded, we need to prove that $\bK$ is bounded away from $0$, for which we use an iterative procedure based on approximate message passing (AMP), a technique now used widely in statistical inference \cite{DMM09,BM11,DAM16,MV21}, and in studying spin glasses \cite{Bol14,DS19,Mon19,CT21,EAMS21,ES22,HS24}.

The proof that the covariance is close to $\bK^{-1}$ relies on the classical \emph{cavity method}---comparing a Gibbs distribution of a Hamiltonian to a version where one site is removed---introduced in statistical physics literature by M\'{e}zard--Parisi--Virasoro \cite{MPV87b}, developed by Aizenman--Sims--Starr \cite{ASS03}, and extensively used thereafter by Talagrand \cite{Tal10,Tal11} among others.

\parhead{Concurrent work by OpenAI.} The problem of sampling from the SK model was one among $722$ claimed results released very recently by OpenAI. In addition to establishing a spectral gap for $\beta < 1$, they prove cutoff from a worst-case initialization, and establish a mixing time of $n^{5/3 + o(1)}$ at $\beta = 1$. We provide a short summary of our interpretation of their proof of a spectral gap \cite{Ope26} below.

They begin essentially by proving a weak \Poincare inequality, albeit with a much better weakness parameter than ours. Using the standard stochastic localization that we do, they show that the covariance is bounded on the localization path with probability $1 - e^{-cN}$, and that the endpoint measure satisfies a \Poincare inequality with probability $1 - e^{-cN}$. We remark that their proof bounding the covariance significantly differs from ours. They show that under a certain ``stability'' criterion, the distribution has bounded covariance.

The aforementioned stability criterion requires local convexity at \emph{every} approximate fixed point of the TAP free energy---this differs from previous results \cite{Cel24}, where it is shown that a late AMP iterate approximates the mean well, and that one has local convexity at this AMP iterate. In particular, their Lemma 5.7 uses this to conclude that there is a unique TAP fixed point.

We remark that such conclusions are expected to be \emph{false} in the $p$-spin model (in the regime with metastable states); it seems likely that our proof is more flexible in this regard, and indeed a similar strategy succeeds at establishing a weak \Poincare inequality for $p$-spin models up to the $\beta_{\mathsf{SL}}$ threshold from \cite[Eq. (1.5)]{HMP24}.

Given this weak \Poincare inequality, the key technical innovation is that they are then able to boost this to a true \Poincare inequality. To do this, they express the \Poincare inequality in a more fine-grained fashion. For the stopping time $\tau$ as in \cite[Lemma 6.5]{HMRW24} (which is equal to $T$ with probability $1 - e^{-cN}$), denote by $\calE$ the event that $\tau \ne T$ or that $\mu_{T}$ does not satisfy a \Poincare inequality---this is an event on the space of trajectories of stochastic localization. One then has  
\[ \calE(f,f) \geq  \rho \cdot \left( \Var[f] - \E[\bone_{\calE} \cdot \Var_{\mu_{\tau}}[f]] \right) \mper \]
Their goal now is to show that $\E\left[ \bone_{\calE} \cdot \Var_{\mu_{\tau}}[f] \right]$ is small relative to $\calE(f,f)$. While their proof as stated does not express it in this way, it may be interpreted as showing that for any function $f$ that has mean $0$ and variance $1$ (under $\mu_0$),
\begin{equation}
    \label{eq:oai-target-probability-transfer}
    \E\left[ \bone_{\calE} \cdot \left\langle f^2 \right\rangle_{\mu_{\tau}} \right] \le o(1) + O(1) \cdot \calE(f,f) \mper
\end{equation}
Given this, one can substitute this bound and rearrange in order to prove a true \Poincare inequality.
In contrast, in our work and in \cite{HMRW24}, we use the throwaway bound $\E\bracks*{ \Ind_{\calE} \cdot \Var_{\mu_{\tau}}[f] } \le \osc(f)^2 \cdot \Pr[\calE]$.

To prove this, they consider the auxiliary stochastic localization process defined as follows. Denote by $\nu_0$ the distribution with density $\mu_0 \cdot f^2$. For some draw $X$ and $Y_{t} = tX + B_{t}$, let $\bbP$ be the law of the posterior on $X$ under the prior $X \sim \mu_0$, defining a measure-valued process $(\mu_t)_{t \ge 0}$. Analogously, let $\bbQ$ be the law of the posteriors under the prior $X \sim \nu_0$. On a common observation path $\calF_t = \sigma\left( (Y_s)_{s \le t} \right)$, Bayes' rule yields that the law of $\nu_t$ is simply that of $\frac{\mu_t \cdot f^2}{\langle f^2 \rangle_{\mu_t}}$, and so for any time $t$,
\begin{equation}
    \label{eq:oai-reweighed-density}
    \frac{\dif \bbQ_{t}}{\dif \bbP_{t}} = \langle f^2 \rangle_{\mu_t} \mper
\end{equation}
Consequently,
\[ \E \left[ \bone_{\calE} \langle f^2 \rangle_{\mu_{\tau}} \right] = \bbQ_{\tau}(\calE) \mper \]
The goal now is to port the established bound $\bbP_{\tau}(\calE) \le e^{-cN}$ over to a useful bound on $\bbQ_{\tau}(\calE)$. To do this, they show that the KL divergence
\[ \KL{\bbQ_{\tau}}{\bbP_{\tau}} = \Ent\left[ \langle f^2 \rangle_{\mu_{\tau}} \right] = \Ent_{\mu_0}[f^2] - \E \left[ \Ent_{\mu_{\tau}}[f^2] \right] \]
is $o(N) + O(N) \cdot \calE(f,f)$.
This divergence can be expressed using It\^{o}'s formula in terms of the trajectory of stochastic localization, which yields
\[ \KL{\bbQ_{\tau}}{\bbP_{\tau}} = \frac{1}{2} \E_{X \sim \nu_0, B} \int_0^\tau \left\| \bary(\nu_t) - \bary(\mu_t) \right\|^2 \mcom \]
where the measures $\mu_t$ and $\nu_t$ are coupled by using the same observation path, as in \eqref{eq:oai-reweighed-density}. The remainder of the proof is in two steps. First, they replace $\bary(\mu_t)$ with the TAP approximation, which is more tractable. The second step, which is more involved, requires going through the proof that the TAP fixed point approximates the mean and tracking how reweighing by $f^2$ influences the proof. In particular, they ``Taylor expand'' $f$ out in terms of its discrete derivatives to order $1$. These discrete derivative terms may then be bounded in terms of the Dirichlet form (and the variance), yielding the second term in \eqref{eq:oai-target-probability-transfer}; the first term arises from errors in the TAP approximation. We refer the reader to \cite[Proposition 7.3]{Ope26} for more details.

\parhead{Organization.} We begin in \Cref{sec:prelims} with an overview of the techniques we will require. \Cref{sec:warmup} provides a self-contained proof of the high-probability zero-field covariance bound \Cref{th:zero-field-cov-bound}, whose approach we build on in the remainder of the paper. \Cref{sec:cov-bound} establishes the high-probability covariance bound \Cref{th:cov-bounds} at nonzero localization time, conditioned on the inverse prediction $\bK$ being positive and bounded away from $0$. \Cref{sec:lower-bd-inv-cov-pred} shows that the inverse prediction $\bK$ is indeed bounded away from $0$ with high probability. In \Cref{app:technical-estimates}, we prove a handful of lemmas omitted from the main body.

\subsection{Future directions}

We propose a few directions for further study.

\parhead{Mixing in SK model with a field.}
An interesting extension of the results in this paper would be to establish mixing in the SK model under a field up to the critical threshold.
\begin{problem}[Mixing under the de Almeida--Thouless line] \label{prob:AT-line-mixing}
    In the presence of a field $h\cdot\bone$ for $h\in\R$, for $\bW\sim\GOE(N)$, prove that simulated annealing samples from the Ising model $\mu_{\beta\bW, h\cdot\bone}$ in polynomial time as long as $0 \le \beta < \beta_{\star}(h)$, where $\beta_{\star}(h)$ is described by the de Almeida--Thouless line from \cite{AT78}.
\end{problem}
A qualitative version of the above is due to \cite{BER26}.
Beyond settling a conjectured spin glass threshold, the above question is of interest in the connection between sampling from the SK model, and the efficacy of Markov chains at optimal inference in the spiked Wigner model via a connection made by \cite{LMRRW24,RW26}.

\parhead{Counting vs.\ sampling in average-case models.}
A conventional wisdom in the classical sampling literature is that the algorithmic task of computing a partition function has roughly the same complexity as sampling from a Gibbs distribution, and there is a formal equivalence \cite{JVV86,SJ89} in that an algorithm for computing the partition function for a family of Hamiltonians gives a sampling algorithm for the corresponding Gibbs distributions.

In the average-case setting, the situation is murkier.
As a running example, consider the spherical $p$-spin glass model, where the Hamiltonian is a degree-$p$ random polynomial whose coefficients are drawn independently from a Gaussian distribution.
It is believed that Markov chain-based methods should fail to sample from the Gibbs distribution when the inverse temperature exceeds the \emph{shattering threshold} $\beta_{\shatter}$, and the best sampling algorithms work up to a constant of the threshold \cite{HMP24,HMRW24}.
On the other hand, an algorithm for computing the partition function was given well beyond $\beta_{\shatter}$ up to the ``replica symmetric threshold'' in a work of \cite{BHLLG26}, which suggests that spin glass models may exhibit phases where sampling is hard, but computing the partition function is tractable.
However, this story was complicated by the recent work \cite{ZSSJ26}, which uses a quantum-inspired algorithm to sample in the shattered phase.

We consider it an interesting direction to understand what the true algorithmic threshold is for sampling, as well as what the threshold is for Markov chains.



\parhead{Certifiable sampling.}
We say that a sampling algorithm for an Ising model is \emph{certifiable} if there is a $\poly(N)$-time verifier that takes as input the interaction matrix of an Ising model, and outputs either \texttt{yes} or \texttt{abstain}, and further outputs \texttt{yes} with high probability over the disorder.
When the verifier outputs \texttt{yes}, the sampling algorithm must additionally (approximately) sample from the target distribution.

Many of the previous results for mixing in Ising models \cite{EKZ22,CE22,AKV24,Wan26} that yielded samplers for the SK model when $\beta < 1/2$ in fact give certifiable samplers since the proofs of mixing rely on spectral conditions on the interaction matrix, or other closely related matrices.
The proof of mixing in this paper relies significantly more on the randomness of the disorder compared to previous works, and hence does not yield an efficiently checkable condition.

\begin{problem}
    What is the threshold $\beta_{\star}$ such that there is a certifiable sampler from the SK model at inverse temperature $\beta < \beta_{\star}$?
\end{problem}

We anticipate that certifiable samplers are a useful primitive for certain downstream applications of sampling to statistical inference and privacy, where it may be essential to be certain that one is indeed sampling from the desired target distribution.
Beyond applications, if $\beta_{\star}$ is indeed equal to $1$, we anticipate that it would yield significantly simpler proofs of the results in this paper.
In a world where $\beta_{\star}$ is \emph{not} equal to $1$, an interesting takeaway is that the proof of success of any sampler for the SK model necessarily strongly exploits the distributional assumption of the input.



\subsection{AI methodology}

This paper is the result of a long-running project on sampling from the high-temperature SK model up to $\beta < 1$ using simulated annealing, starting early 2025. Our original approach was centered on the TAP Hessian $\bK$, and tried to show that for constant $p$, one has the bound $\E \Tr\left( \left( Z \cdot (\bK \Cov - \Id) \right)^{2p} \right) \le \frac{1}{N^{p-1}} \cdot (\E Z)^{2p}$, by manually expanding out this expectation and witnessing various cancellations. At nonzero localization time, this would further involve ``TAP planting'' in order to have a tractable approximation to the mean.

We established this in the zero-field case for $p = 1$, and for $p=2$, used a computer program to bound various lower-order terms in this expansion (for context, we would require $p\ge 12$ in order to get meaningful guarantees for simulated annealing). We shortly discovered that this expectation is not bounded all the way up to $\beta = 1$, due to low-probability tail events. Nevertheless, we had a good sense of the kind of high-probability event that one might be able to condition on in order to get a useful high-probability covariance bound using this proof.

In mid-2026, we tried using a higher-order prediction for the covariance matrix using the Plefka expansion \cite{Ple82,GY91}: this is a more precise prediction $\bK^{(\ell)}$ for the inverse covariance featuring higher-order corrections, and we tried showing that $\E \left\| Z \cdot (\bK^{(\ell)} \Cov - \Id) \right\|_{\mathrm F}^2 \le \frac{1}{N^{\ell-1}}$. This plan was successfully executed for some small $\ell$, but it was not sufficient for the simulated annealing application.

At the launch of GPT-5.5, we attempted using it to establish the bounds described above. These attempts were not fruitful, and produced nothing beyond the earlier results that had been proved by hand. This process was repeated with the launch of GPT-5.6 Sol, which produced a 150-page document: this document featured a high-probability zero-field covariance bound using the Plefka expansion down to an arbitrarily small inverse polynomial error probability using the previously sketched approaches, but failed to prove similar results in the non-zero field case.

When GPT-6 Astra Pro launched, all the previous documents were fed in with a proof strategy, and it produced an $85$-page document with a complete proof of a weak \Poincare inequality strong enough to yield guarantees for simulated annealing. However, this proof deviated from the prompt, and regressed from using the Plefka expansion to the earlier idea of $\Tr\left( \left( Z \cdot \left(\bK \Cov - \Id\right) \right)^{2p} \right)$ after conditioning on a high-probability event, in addition to using the so-called ``$1/N$-expansion,'' e.g. \cite{CGTvH26}. Furthermore, at nonzero localization time, it continued to use the recommended approach of TAP planting to have a tractable approximation of the mean. Various parts of this proof were digested over the following week. We note that this proof did successfully establish that the covariance prediction is $O\left( \frac{1}{\sqrt{N}} \right)$-close to the covariance, a bound that the present manuscript does \emph{not} establish. In a nutshell, this proof considered two expansions of the inverse covariance, one in $\beta$ and the other in $1/N$; showing the well-definedness of this requires establishing analyticity of certain functions. It then analyzed various terms in the $\beta$-expansion, then transferred these bounds over to useful statements for the coefficients of the $1/N$-expansion.

After digestion, the authors provided suggestions to simplify the proof based on their understanding of the zero-field proof, attempting to take more advantage of the aforementioned analyticity. This produced a $33$-page document containing a proof similar to that in the present manuscript (which, notably, does not use this analyticity). Over the subsequent weeks, the authors digested this proof and simplified numerous parts of it, resulting in the present document: these efforts utilized many incarnations of GPT-6 Astra, ranging in power from Medium to Ultra.

GPT-6 Astra was additionally used in typesetting, the detection of bugs, and finding references. Tautologically, the authors take all responsibility for this manuscript.




\section{Preliminaries}
\label{sec:prelims}

\subsection{Notation}

\begin{itemize}
	\item We use $C$ to denote appropriately large universal constants whose value may change from line to line, and $c$ to denote similarly fickle small constants.
	\item Given a symmetric matrix $M$, we denote by $\opnorm{M}$ or $\|M\|$ its largest singular value, and by $\|M\|_{\mathrm F}$ its Frobenius norm. 
	\item Given a Gibbs measure $\mu$ over $\{\pm 1\}^N$, we denote by $\langle f(\sigma)\rangle_{\mu}$ the Gibbs average of $f$ with respect to the measure $\mu$. When the measure is clear from context, we denote it simply as $\langle f(\sigma)\rangle$ or $\langle f\rangle$.
	\item Given tensors $M \in (\R^{N})^{\otimes s}$ and $K \in (\R^{N})^{\otimes t}$ for some integer $t \le s$, we denote by $M[K] \in (\R^{N})^{\otimes (s-t)}$ the tensor
	\[ (M[K])_{i_1,\ldots,i_{s-t}} = \sum_{i_{s-t+1},\ldots,i_{s}} M_{i_1,\ldots,i_s} \cdot K_{i_{s-t+1},\ldots,i_{s}} \mper \]
\end{itemize}

\subsection{Simulated annealing}
In this section, let $H:\{\pm1\}^N \to \R$ be a Hamiltonian, and let $\mu_{\beta H}(x)\propto \exp\parens*{ \beta H(x) }$ be the corresponding Gibbs distribution.

We begin by defining Glauber dynamics.
\begin{definition}[Glauber dynamics]
    We describe the \emph{Glauber dynamics} Markov chain for sampling from $\mu_{\beta H}$ via a single step of the Markov chain.
    \begin{displayquote}
        Let $x$ be the current state.
        Pick $i\sim[N]$ uniformly at random, and transition $x'\sim\mu_{\beta H}|x_{-i}$ where $x_{-i}$ refers to $x$ with the $i$-th coordinate erased.
    \end{displayquote}
\end{definition}
In the simulated annealing algorithm, one runs Glauber dynamics but changes $\beta$ according to a schedule.
\begin{definition}[Simulated annealing]
    Let $(\beta_i)_{i\in[T]} > 0$ be a sequence of inverse temperatures and $x\in\{\pm1\}^N$ be an initialization.
    The simulated annealing algorithm with schedule $(\beta_i)_{i\in[T]}$ runs $T$ steps of Glauber dynamics, where the $t$-th step is taken according to Glauber dynamics for $\mu_{\beta_t H}$.
\end{definition}

We rely on the fact that a weak \Poincare inequality for $\mu_{\beta H}$ implies that simulated annealing samples; see, e.g., \cite[Theorem 4.13]{HMRW24}.
\begin{lemma}[Simulated annealing from weak \Poincare inequalities]
    \label{lem:wpi-annealing}
    Let $0=\beta_0<\beta_1<\cdots<\beta_L=\beta$.
    Suppose that each $\mu_{\beta_j H}$ satisfies a
    $(\rho,\eps)$-weak \Poincare inequality.
    Suppose for each $j\in[L]$ and some sufficiently small constant $c > 0$:
    \[
        \beta_j - \beta_{j-1} \le \frac{c}{H_{\max} - H_{\min}}\mcom
    \]
    Starting from a uniformly random initialization, run $m$
    Glauber updates at each inverse temperature $\beta_j$.
    Then the output distribution $\nu$ satisfies
    \[
        \dtv{\nu}{\mu_{\beta H}}
        \le
        O\left( L \left( e^{-\frac{m}{2\rho}}+\sqrt{\eps} \right) \right) \mper
    \]
\end{lemma}

\subsection{Stochastic localization}

To establish a weak \Poincare inequality, we use stochastic localization (see, e.g, \eqref{eq:sl-sde}) to (approximately) decompose the Gibbs distribution into measures that each satisfy a \Poincare inequality.
In particular, we use the following consequence of \cite[Lemmas 6.8 and A.6]{HMRW24}.
\begin{lemma}[Localization criterion]
    \label{lem:sl-wpi}
    Let $(\mu_t)_{0\le t\le T}$ be stochastic localization
    initialized at a fixed distribution $\mu$ on $\{\pm1\}^N$.
    Suppose that, with probability at least $1-\delta$, the following two events hold:
    \begin{itemize}
        \item $\sup_{0\le t\le T} \|\Cov(\mu_t)\|_{\mathrm{op}}\le K$.
        \item $\mu_T$ satisfies a \Poincare inequality with constant $\rho_T$.
    \end{itemize}
    Then $\mu$ satisfies a $(\rho_Te^{KT},C\delta e^{KT})$-weak \Poincare inequality.
\end{lemma}
Thus, it remains to control the covariance along the localization
path.
To obtain a handle on the distributions encountered along a localization path, we will shift to the so-called ``planted'' model in order to make the localization process more tractable.

\begin{lemma}
	\label{lem:planted-contiguity}
	Let $0 \le \beta < 1$. Let $\mu_{\nullmodel}$ be the law over measures defined by a Hamiltonian $\bH_{\nullmodel}$ obtained by drawing $\bW \sim \GOE(N)$ and setting $H_{\nullmodel}(\sigma) = \frac{\beta}{2} \cdot \sigma^\top \bW \sigma$. Let $\mu_{\pl}$ be the law over measures defined by a Hamiltonian $\bH_{\pl}$ obtained by drawing $\bW \sim \GOE(N)$ and $\bx \sim \{\pm 1\}^N$, then setting $\bH_{\pl}(\sigma) = \frac{\beta}{2} \cdot \sigma^\top \bW \sigma + \frac{\beta^2}{2N} \cdot \langle \sigma,\bx\rangle^2$. Then, $\mu_{\nullmodel}$ and $\mu_{\pl}$ are mutually contiguous. More quantitatively, for any event $\calE$ over the space of measures on $\{\pm 1\}^N$,
	\[ c \mu_{\pl}(\calE)^2 \le \mu_{\nullmodel}(\calE) \le C \sqrt{\mu_{\pl}(\calE)} \mper \]
\end{lemma}
\begin{proof}
	Let us focus on the second inequality. We have
	\[ \mu_{\nullmodel}(\calE) = \E_{\pl} \left[ \bone_{\calE} \cdot \frac{\dif \mu_{\nullmodel}}{\dif \mu_{\pl}} \right] \le \sqrt{ \mu_{\pl}(\calE)  } \cdot \sqrt{ \E_{\pl} \left[ \left(\frac{\dif \mu_{\nullmodel}}{\dif \mu_{\pl}}\right)^2 \right] } = \sqrt{ \mu_{\pl}(\calE)  } \cdot \sqrt{ \E_{\nullmodel} \left[ \frac{\dif \mu_{\nullmodel}}{\dif \mu_{\pl}} \right] } \mper \]
	It remains to show that second quantity is $O(1)$. Note now that for any prescribed interaction matrix $W$ (without the scaling by $\beta$),
	\begin{align*}
		\frac{\dif \mu_{\pl}}{\dif \mu_{\nullmodel}}(W \mid \bx = x) &= \frac{ \exp\left( - \frac{N}{4} \left\| W - \frac{\beta}{N} xx^\top \right\|_{F}^2 \right) }{{ \exp\left( - \frac{N}{4} \|W\|_{F}^2 \right) }} \propto \exp\left( \frac{\beta}{2} x^\top W x \right) \mcom \text{so} \\
		\frac{\dif \mu_{\pl}}{\dif \mu_{\nullmodel}}(W) &\propto \sum_{x \in \{\pm 1\}^N} \exp\left( \frac{\beta}{2} x^\top W x \right) = Z_{\beta W} \mper
	\end{align*}
	Our goal is then simply to show that
	\[ \E \left[\frac{1}{Z_{\beta\bW}}\right] \lesssim \frac{1}{\E[Z_{\beta\bW}]} \mper \]
	This estimate is easily obtained by integrating the lower tail described in \cite[Theorem 11.2.1]{Tal11}. The first inequality is similar and requires us to show that $\E[Z_{\beta\bW}^2] \lesssim \E[Z_{\beta\bW}]^2$, which is far easier to establish.
\end{proof}




By the characterization of stochastic localization given by El Alaoui and Montanari \cite{EM22}, along with contiguity \Cref{lem:planted-contiguity}, and the fact that $(\bx,\bH_{\pl})$ has the same law as sampling $\bH_{\pl}$ and then sampling $\bx$ from $\mu_{\beta\bM}$ for $\bM\coloneqq \bW + \beta\frac{\bx\bx^{\top}}{N}$, it suffices to study the covariance matrix of the Ising model $\mu_t \coloneqq \mu_{\beta \bM, t\bx+\bB_t}$.

\thcovbounds*

We now describe how to complete the proof of the main theorem.
\begin{proof}[Proof of \Cref{th:main}]
    We need the following stronger version of \Cref{th:cov-bounds} and prove that for any constant $a > 0$:
    \[
        \Pr_{\bW,\bx}\bracks*{ \Pr_{(\mu_t)_{t\in[0,T]}\sim\mathrm{SL}}\bracks*{ \sup_{0\le t \le T} \left\| \Cov\parens*{\mu_t} \right\|_{\mathrm{op}} \le K } \ge 1-O(N^{-a}) } \ge 1-O(N^{-a})\mper
        \numberthis \label{eq:full-path}
    \]
    The proof of this closely follows that of \cite[Lemma 6.7]{HMRW24}, which is based on using \Cref{th:cov-bounds} to obtain a high probability bound for a given time, and then taking a union bound over a net of times.

    With \eqref{eq:full-path} in hand, we can use \Cref{lem:sl-wpi} to obtain a weak \Poincare inequality for each temperature, take a union bound over temperatures, and apply \Cref{lem:wpi-annealing} to complete the proof.
\end{proof}


\section{Warmup: High-probability covariance bounds for the zero-field SK model}
\label{sec:warmup}

	In this section, we sketch some ideas in the proof of \Cref{th:cov-bounds} in the setting where $h = 0$.

	\zerofieldcovbound*

    Since the diagonal entries of $\bW$ do not affect the covariance, we may set these entries to $0$ throughout this section.
	Define
	\[ K = (1+\beta^2) \Id - \beta \bW \mper \]
    Note that $K$ is PSD and bounded away from $0$ with high probability, as $\lambda_{\max}(\bW)$ converges in probability to $2$.

	We start by using the ``cavity formula'' to express the covariance. Given an external field $h \in \R^n$, denote by $m(h)$ the mean of the SK model with interaction matrix $\bW$ and external field $h$. Similarly, given $i \in [N]$, denote by $m^{(i)}(h)$ the mean of the SK model (on $N-1$ spins) with the $i$th spin removed, with interaction matrix $\bW_{-i,-i}$ and external field $h$. Similarly denote by $\Cov^{(i)}(h) \in \R^{n \times n}$ the covariance matrix of the cavity model, with the $i$th row and column equal to $0$. In particular, we denote $\Cov^{(i)} = \Cov^{(i)}(\mathbf{0})$. For ease of notation, let us also define
    \[ d^i = m^{(i)}(\beta\bW_i) \text{ and } D^i = \Cov^{(i)}(\beta\bW_i) \mper \]
    These are respectively the mean and covariance in the cavity model where the $i$th spin is pinned to $+1$. The law of total covariance then yields
    \[ \Cov = \begin{pmatrix} 1 & (d^i)^\top \\ d^i & D^i + d^i (d^i)^\top \end{pmatrix} \mper \]
    We may expand out the matrix $K \Cov - \Id$ using the above cavity expansion---see \Cref{sec:cov-bound} for more details, where the analogue of this decomposition is described for nonzero localization time.

	\begin{lemma}\label{lem:warmup-cov}
		For each row $i$, we may write
		\[ (K \Cov - \Id)_{i} = \alpha^{i} \Cov_{i} + Y_{i} - Z_{i} \mcom \]
		where $Y_{i,i} = Z_{i,i} = 0$,  $\Cov^{(i)} = \Cov^{(i)}(0)$, and
		\begin{align*}
			\alpha^i &= \beta^2 - \beta \langle \bW_i , d^i \rangle \\
			Y_{i,-i} &= (d^i - \beta \Cov^{(i)} \bW_{i})^\top \\
			Z_{i,-i} &= \beta \bW_{i}^\top \left( D^i - \Cov^{(i)} \right) \mper
		\end{align*}
	\end{lemma}


	Our next goal will be to show that $\alpha^{i} = o(1)$, $\|Y\|_{F}^2 = O(1)$, and $\|Z\|_{F}^2 = O(1)$ with high probability. This would complete the proof, since we then have
	\[ \cdot \Cov =  (K - \diag(\alpha))^{-1} \left( \Id + Y - Z \right) \psdle O(1) \mper \]

	We first control $\|Y\|_F^2$.
    The idea is to Taylor expand the function $m^{(i)}$ defining $d^i$, which yields an accurate approximation because the external field $\beta\bW_i$ is small (with entries of order $N^{-1/2}$).
    In particular, since $\Cov^{(i)}(h) = \grad_h m^{(i)}(h)$, the first order Taylor approximation around $h=0$ is
    \[
        d^i \approx \beta \Cov^{(i)} \bW_i,
    \]
    which is exactly the other term in $Y_{i,-i}$.
    We can control the error in this approximation by Taylor expanding further; the error depends on the next nonzero term of the expansion.
    Before carrying out the expansion, we note that the terms in the Taylor expansion take the form of cumulants of the partition function.
    Let $Z^{(i)}(h)$ be the SK partition function at external field $h$, where spin $i$ and its interactions are removed.
    Let $\kappa_r^{(i)}(h) = \grad_h^r \log Z^{(i)}(h)$.
    It is standard that $\kappa_1^{(i)}(h) = m^{(i)}(h)$.
    If we write $\kappa_r^{(i)} = \kappa_r^{(i)}(0)$, then (because $\kappa_1^{(i)} = \kappa_3^{(i)} = 0$)
    \[ 
        d^i \approx \sum_{r=1}^{3} \frac{\beta^{r-1}}{(r-1)!} \kappa_r^{(i)}[\bW_i^{\otimes (r-1)}] 
        = \beta \Cov^{(i)} \bW_i\mcom 
    \]
    with squared $\ell^2$ error of order $\|\kappa_4^{(i)}(h)[\bW_i^{\otimes 3}]\|_2^2$. We discuss the bounding of this error shortly---this will be the complicated part of the proof.




    We can bound $\|Z\|_F^2$ near-identically, comparing $D^i = \Cov^{(i)}(\beta\bW_i)$ to $\Cov^{(i)} = \Cov^{(i)}(0)$ by Taylor expansion of $\Cov^{(i)} = \kappa_2^{(i)}$.
    We similarly find that each row has norm $\|Z_i\|^2 = O(1/N)$ with very high probability, and so $\|Z\|_F^2 = O(1)$.


	The fact that $\alpha^i = o(1)$ also follows by Taylor expanding $d^i$ as above.
    We have
	\[ \langle \bW_i , d^i\rangle \approx \beta \Cov^{(i)}[\bW_i^{\otimes 2}] \mcom \]
    with error of order $|\kappa_4^{(i)} [\bW_i^{\otimes 4}]|$.
    We finally have that with high probability, 
    \[ \Cov^{(i)}[\bW_i^{\otimes 2}] = \frac{\Tr \Cov^{(i)}}{N} + o(1) = 1 + o(1), \]
    because, by overlap concentration, the conditional variance of this random variable is bounded by
    \[
        \Var\left[ \Cov^{(i)}[\bW_{i}^{\otimes 2}] \Big| \bW_{-i,-i} \right] \le \frac{2\|\Cov^{(i)}\|_{F}^2}{N^2} \le O\left( \frac{1}{N} \right)
    \]
    with high probability, and its conditional expectation is exactly $\frac{1}{N}\Tr \Cov^{(i)}$.
    Gaussian concentration can then be used to obtain an error probability bound of $e^{-c\sqrt{N}}$: we have for a sufficiently small constant $\eps$ that
    \[ \Pr\left[ \left| \Cov^{(i)}[W_i^{\otimes 2}] - \frac{\Tr \Cov^{(i)}}{N} \right| > \eps \right] \le 2 \exp\left( - c \min\left\{ \frac{\eps^2 N^2}{\|M\|_{\mathrm F}^2} , \frac{\eps N}{\opnorm{M}} \right\} \right) \le e^{-c\sqrt{N}} \mcom \]


    Now, how do we show thaat the error $\|\kappa_4[\bW_i^{\otimes 3}] \|_2^2$ to be small.
    We may establish this by controlling its high moments using Wick's formula, which various produces terms involving trace bounds and other cumulants. For example, one has
    \[ \E_{\bW_i} \left\| \kappa_{4}[\bW_i^{\otimes 3}] \right\|^2 = \frac{1}{N^3} \cdot \left( 6 \|\kappa_4\|_{\mathrm F}^2 + 9 \| \underbrace{\kappa_4[\Id]}_{= - 2 \Cov^2} \|_{\mathrm F}^2 \right) \mcom \]
    and overlap concentration allows us to bound both of these terms by $O(N^2)$ (in general $\|\kappa_{r}\|_{\mathrm F}^2 \le N^{r/2}$), and thus bound the expectation by $O(1/N)$. However, bounding constant moments in this way and using overlap concentration will only allow us to establish that for constant $p$, $\opnorm{\Cov} = O_{L^p}(1)$: this only yields a high-probability bound on the norm of the covariance by $O(N^{\delta})$.

    Nevertheless, this already suffices to get better bounds on the second term in the expectation. Indeed, earlier we bounded $\| \Cov^2 \|_{\mathrm F} \le \| \Cov \|_{\mathrm F}^2 = O_{L^p}(N)$, but we may now bound  $\| \Cov^2 \|_{\mathrm F} \le \| \Cov \|_{\mathrm F} \cdot \opnorm{\Cov} = O_{L^p}(\sqrt{N})$, which yields a high-probability bound for the second term. Controlling the other term requires us to bound $\| \kappa_4 \|_{\mathrm F}^2$ by something better than $O(N^2)$; an improved bound of this form will also suffice to obtain the necessary better bounds on higher moments. This will require some more ideas, and we defer a detailed discussion to the actual proof, which is presented in \Cref{subsec:proof-of-improved-cumulant}.


\section{High-probability covariance bounds at nonzero localization time}
\label{sec:cov-bound}

	In this section, we establish \Cref{th:cov-bounds}, restated for convenience.

	\thcovbounds*

	We will use the convention that the diagonal of $\bM$ is zeroed out.

	We now establish some notation.
	Set:
	\[
		m = \langle \sigma \rangle \mcom \qquad v_i = 1-m_i^2 \mcom \qquad q_N = \frac{\norm*{m}^2}{N}\mcom
	\]
	and define the TAP Hessian
	\begin{equation}
		\label{eq:inv-cov-pred}
		K \coloneqq \diag(v_i^{-1}) - \beta \bM + \beta^2(1-q_N)\Id - \frac{2\beta^2}{N}mm^{\top} \mper
	\end{equation}
	We will bound $K\Cov-\Id$ with high probability and combine this with a lower bound on the smallest eigenvalue of $K$.

	\begin{lemma}[Operator norm of the residual]
	\label{lem:organized-residual}
	For any constant $a > 0$, with probability $1 - O(N^{-a})$, we have
	\[
		\|K\Cov-\Id\|_{\mathrm{op}}\le O(1).
	\]
	\end{lemma}

	\begin{lemma}[Smallest eigenvalue of the TAP Hessian]
	\label{lem:organized-gap}
	There is a constant $\gamma>0$, depending only on
	$\beta$ and $T$, such that for any constant $a > 0$,
	\[
		\Pr\left[ K \not\psdge \gamma \right] \le N^{-a} \mper
	\]
	\end{lemma}

	\begin{corollary}[Covariance bound]
	\label{cor:organized-covariance}
	For some constant $L$, for any constant $a > 0$, we have
	\[ \Pr\left[ \opnorm{\Cov} \ge L \right] \le N^{-a} \mper \]
	\end{corollary}

	The proof of \Cref{lem:organized-residual} depends on the following estimates.
	Let $f_i \coloneqq (K\Cov-\Id)_{i,-i}$, and let $\lambda_i\coloneqq \beta^2(1-q_N) - \beta d_i^\top \bM_{i,-i}$ where $d_i \coloneqq \frac{\Cov_{i,-i}}{v_i}$.
	\begin{lemma}	\label{lem:organized-off-diagonal}
		For every constant $a > 0$, there is a constant $C_a$ such that
		\[
			\Pr\bracks*{ \sum_{i\in[N]} \norm*{ f_i - \lambda_i \Cov_{-i,i} }_2^2  > C_a } \le N^{-a}
		\]
	\end{lemma}
	We prove the following estimate for the diagonal part.
	\begin{lemma}[Diagonal terms]
	\label{lem:organized-diagonal}
	We have
	\(
		(K\Cov-\Id)_{ii}
		=\lambda_i v_i-\frac{2\beta^2m_i}{N}(\Cov m)_i\mper
	\)
	Further, for every $a>0$ we have:
	\[
		\Pr\left[\max_{i\in[N]}|\lambda_i|>N^{-1/4}\right]
		\le O(N^{-a}).
	\]
	Hence,
	\[
		\max_{i\in[N]}
		\left|(K\Cov-\Id)_{ii}-\lambda_i v_i\right|
		\le 2\beta^2.
	\]
	\end{lemma}
	We defer the proofs of \Cref{lem:organized-off-diagonal,lem:organized-diagonal} to \Cref{sec:proof-cov-bound}, and prove \Cref{lem:organized-residual} below.
	\begin{proof}[Proof of \Cref{lem:organized-residual}]
		We begin by conditioning on the high probability events from
		\Cref{lem:organized-off-diagonal,lem:organized-diagonal,lem:organized-gap}, which simultaneously hold with probability $1-O(N^{-a})$.
		Now define
		\[
			D=\diag(\lambda_i),
			\qquad R=(K-D)\Cov-\Id.
		\]
		Then we have
		\[
			\left\|R-\diag(R)\right\|_F^2 = \sum_{i=1}^N \|f_i - \lambda_i \Cov_{-i,i}\|_2^2 \le C\mper
		\]
		By the identity in \Cref{lem:organized-diagonal} and the fact that $\Cov_{ii}=v_i$, we have for all $i$ that
		\[ |R_{ii}| = \left| (K \Cov - \Id)_{ii} - (D \Cov)_{ii} \right| \le \left| (\lambda_i v_i + 2\beta^2) - (\lambda_i (1-m_i^2)) \right| = 2\beta^2 \mcom \]
		and consequently,
		\[
			\|R\|_{\mathrm{op}}
			\le
			\|R-\diag(R)\|_F
			+\max_i|R_{ii}|
			\le O(1)\mper
		\]
		Noting now that $K \psdge \gamma$ for some constant $\gamma$, and $D \psdle o(1)$, this implies that $\opnorm{\Cov} = O(1)$, and thus \Cref{lem:organized-residual,cor:organized-covariance}. 
	\end{proof}
	
	\subsection{Proof of \Cref{lem:organized-off-diagonal}}
    \label{sec:proof-cov-bound}
    We begin by establishing some notational conventions.
    \begin{itemize}
        \item The gradient $\grad$ is always taken exclusively with respect to the external field.
        In particular, for a function $F$, we will use $\grad F(g)$ to denote $\grad_h F(h)|_{h=g}$.
        \item We will use the symbol $\Delta$ exclusively to denote remainder terms arising from Taylor expansions of various functions, with the subscript denoting the function that the term comes from.
    \end{itemize}

		Fix a site $i$, and consider the distribution $\mu^{(i)}$ obtained by removing its field and all its interactions with the other sites.
		Let $u_i\in\R^{N-1}$ be the field on the remaining sites
        , and let $F^{(i)}(u_i)$ be its free energy.
		Write
		\[
			m^{(i)}=\grad F^{(i)}(u_i) \mcom \qquad
			\Cov^{(i)}=\grad^2F^{(i)}(u_i) \mcom \qquad
			\kappa_r^{(i)}=\grad^rF^{(i)}(u_i) \mcom
		\]
		and
		\[
			m^{(i,\pm)} = \grad F^{(i)}(u_i \pm \beta\bM_{i,-i}), \qquad \Cov^{(i,\pm)} = \grad^2 F^{(i)}(u_i \pm \beta \bM_{i,-i})\mper
		\]
		Denote
		\begin{align*}
			\ol{\Cov_i}
			&\coloneqq\frac{1+m_i}{2}\Cov^{(i,+)}
				+\frac{1-m_i}{2}\Cov^{(i,-)}\mcom
        \end{align*}
        and observe that
        \begin{align*}
			d_i &=\frac{\Cov_{i,-i}}{v_i}= \frac{m^{(i,+)}-m^{(i,-)}}{2}\mper
		\end{align*}
		By the law of total covariance we obtain the block decomposition
		\begin{equation}
			\Cov=
			\begin{pmatrix}
			v_i & v_i d_i^\top\\
			v_i d_i & \ol{\Cov_i}+v_i d_i d_i^\top
			\end{pmatrix},
			\label{eq:organized-blocks}
		\end{equation}
		where the $i$th coordinate is written first.
		Denoting $\alpha = \beta^2 (1-q_N)$, observe that
		\begin{equation}
		\begin{split}
		f_i={}&\underbrace{d_i}_{(\operatorname{diag}(v^{-1})\Cov)_{i,-i} }
		-\underbrace{\beta\bigl(\ol{\Cov_i}+v_id_id_i^\top\bigr)\bM_{i,-i} }
					_{(\beta \bM\Cov)_{i,-i}}\\
		&+\underbrace{\alpha v_i d_i}_{(\alpha\Cov)_{i,-i}}
		-\underbrace{\frac{2\beta^2m_i}{N}(\Cov \cdot m)_{-i}}
					_{((2\beta^2/N)mm^\top\Cov)_{i,-i}}.
		\end{split}
		\label{eq:organized-four-terms}
		\end{equation}
		Rearranging the above using \eqref{eq:organized-blocks} and plugging in the definition of $\lambda_i$,
		\[
			f_i - \lambda_i \Cov_{i,-i} = d_i - \beta \ol{\Cov_i}\bM_{i,-i} - \frac{2\beta^2 m_i}{N}(\Cov\cdot m)_{-i}
		\]
        By the Taylor series expansion for $m^{(i,\pm)}$ and $\Cov^{(i,\pm)}$, we have:
        \begin{align*}
            m^{(i,\pm)}&=\grad F^{(i)}(u_i\pm \beta \bM_{i,-i})=m^{(i)}\pm \beta\Cov^{(i)}\bM_{i,-i}+\frac{\beta^2}{2}\kappa_3^{(i)}[\bM_{i,-i}^{\otimes 2}]\pm \frac{\beta^3}{6}\kappa_4^{(i)}[\bM_{i,-i}^{\otimes 3}]+\Delta_{m,i,\pm},\\
            \Cov^{(i,\pm)}&=\grad^2 F^{(i)}(u_i\pm \beta \bM_{i,-i})=\Cov^{(i)}\pm \beta\kappa_3^{(i)}[\bM_{i,-i}]+\frac{\beta^2}{2}\kappa_4^{(i)}[\bM_{i,-i}^{\otimes 2}]+\Delta_{\Cov,i,\pm},
        \end{align*}
        we obtain the corresponding Taylor expansion of $\ol{\Cov_i}$ and $d_i$ as 
		\begin{align*}
			d_i
			=\frac{m^{(i,+)}-m^{(i,-)}}{2}&=\beta\Cov^{(i)}\bM_{i,-i}
				+\frac{\beta^3}{6}\kappa_4^{(i)}[\bM_{i,-i}^{\otimes 3}]+\Delta_{d,i}\mcom\label{eq:taylor-m} \\
			\ol{\Cov_i}
			=\frac{1+m_i}{2}\Cov^{(i,+)}
				+\frac{1-m_i}{2}\Cov^{(i,-)}&=\Cov^{(i)}
				+\beta m_i\kappa_3^{(i)}[\bM_{i,-i}]
				+\frac{\beta^2}{2}\kappa_4^{(i)}[\bM_{i,-i}^{\otimes 2}]
				+\Delta_{\Cov,i}\mper
		\end{align*}
    
		By the above,
		\begin{align*}
			d_i-\beta\ol{\Cov_i}\bM_{i,-i}
			&=\left(
				\beta\Cov^{(i)}\bM_{i,-i}
				+\frac{\beta^3}{6}\kappa_4^{(i)}\left[\bM_{i,-i}^{\otimes 3}\right]
			\right)\\
			&\quad-\left(
				\beta\Cov^{(i)}\bM_{i,-i}
				+\beta^2m_i\kappa_3^{(i)}\left[\bM_{i,-i}^{\otimes 2}\right]
				+\frac{\beta^3}{2}\kappa_4^{(i)}\left[\bM_{i,-i}^{\otimes 3}\right]
			\right)+\Delta_{f,i}\\
			&=-\beta^2m_i\kappa_3^{(i)}\left[\bM_{i,-i}^{\otimes 2}\right]
				-\frac{\beta^3}{3}\kappa_4^{(i)}\left[\bM_{i,-i}^{\otimes 3}\right]
				+\Delta_{f,i}\mcom
		\end{align*}
        where $\Delta_{f,i}=\Delta_{d,i}-\beta\Delta_{\Cov,i}\bM_{i,-i}$. 
		Therefore
		\begin{equation}
		\begin{split}
		f_i={}&\lambda_i\Cov_{-i,i}
			-\frac{\beta^3}{3}\kappa_4^{(i)}\left[\bM_{i,-i}^{\otimes 3}\right]
			-\beta^2m_i\kappa_3^{(i)}\left[\bM_{i,-i}^{\otimes 2}\right]-\frac{2\beta^2m_i}{N} (\Cov\cdot m)_{-i}+\Delta_{f,i}.
		\end{split}
		\label{eq:organized-cumulant-row}
		\end{equation}
		We may write:
		\[
			\kappa_3^{(i)}\bracks*{ \bM_{i,-i}^{\otimes 2} } = \frac{1}{N}\kappa_3^{(i)}[\Id] + \kappa_3^{(i)}\bracks*{ \bM_{i,-i}^{\otimes 2} - \frac{1}{N}\Id }\mper
		\]
		We will show that the first term of the above nearly cancels the fourth term of \eqref{eq:organized-cumulant-row}, and that the second term is negligible.

		Observe that
		\begin{align*}
			\frac{1}{N}\kappa_3^{(i)}[\Id] &= \frac{1}{N} \grad \Tr\Cov^{(i)} 
			= \frac{1}{N}\grad \parens*{ N-1 - \norm*{m^{(i)}}_2^2 } 
			= -\frac{2}{N}\Cov^{(i)}\cdot m^{(i)}\mper
		\end{align*}
		Plugging the above into \eqref{eq:organized-cumulant-row} gives
		\begin{align*}
			f_i-\lambda_i \Cov_{-i,i} = &\frac{2\beta^2 m_i}{N} \parens*{ \Cov^{(i)}\cdot m^{(i)} - (\Cov\cdot m)_{-i} } \\
			&-\beta^2 m_i \kappa_3^{(i)}\bracks*{ \bM_{i,-i}^{\otimes 2} - \frac{1}{N}\Id } - \frac{\beta^3}{3}\kappa_4^{(i)}\bracks*{\bM_{i,-i}^{\otimes 3}} + \Delta_{f,i}\mper
		\end{align*}
		The below estimates \Cref{lem:error-bounds,lem:organized-taylor-errors}, whose proofs we defer to \Cref{app:technical-estimates}, complete the proof.

		\begin{restatable}{lemma}{errorboundsone}
			\label{lem:error-bounds}
			For every constant $a > 0$, the following bounds hold for all $i\in[N]$ with probability $1-O_{\eps}(N^{-a})$:
			\begin{align*}
				\left\|{ \frac{2\beta^2 m_i}{N} \parens*{ \Cov^{(i)}\cdot m^{(i)} - (\Cov\cdot m)_{-i} } }\right\|_2 &\le \frac{C}{\sqrt{N}} \\
				\left\|{ \frac{\beta^3}{3}\kappa_4^{(i)}\bracks*{\bM_{i,-i}^{\otimes 3}} }\right\|_2 &\le \frac{C}{\sqrt{N}} \\
				\left\|{ \beta^2 m_i \kappa_3^{(i)}\bracks*{ \bM_{i,-i}^{\otimes 2} - \frac{1}{N}\Id } }\right\|_2 &\le \frac{C}{\sqrt{N}} \\
				\max_{i\in[N]}|\lambda_i|&\le N^{-1/4}.
			\end{align*}
		\end{restatable}

		\begin{restatable}{lemma}{errorboundstwo}
			\label{lem:organized-taylor-errors}
			For any $a,\eps > 0$, with probability $1-O(N^{-a})$, we have
			\[
				\norm*{\Delta_{d,i}}_2 \le O\parens*{N^{-1+\eps}},\qquad
				\norm*{\Delta_{\Cov,i}\bM_{i,-i}}_2 \le O\parens*{N^{-3/4+\eps}}\mper
			\]
            and consequently, $\|\Delta_{f,i}\|_2\le \|\Delta_{d,i}\|_2+\beta\|\Delta_{\Cov,i}\bM_{i,-i}\|_2\le O(N^{-3/4+\varepsilon})$.
		\end{restatable}

	\subsection{Overlap concentration and crude bounds on the cumulants}

		In this section, we establish overlap concentration for the tilted model, and use this to establish some crude bounds on the Frobenius norms of the cumulants.

		\begin{lemma}
			\label{lem:shifted-model-overlap-conc}
			Let $-1 \le s \le 1$ fixed, and $t > 0$ bounded. Consider the Ising model $\mu$ with interaction matrix $\beta\bM$, for $\bM = \bW + \frac{\beta}{N} \bx\bx^\top$ (with $\bx \sim \{\pm 1\}^{N}$), and external field
			\begin{equation}
				\label{eq:real-ext-field-law-cavity}
				\bh = \left( t + \frac{s\beta^2\bs}{N} \right) \cdot \bx + \bg \sqrt{t + \frac{s^2\beta^2}{N}} \mcom
			\end{equation}
			with $\bs \sim \{\pm 1\}$ and $\bg \sim \calN\left( 0 , \Id \right)$. Then, for some explicit deterministic constant $q$, we have
			\[ \E \left\langle \left| R_{12} - q \right|^{2k} \right\rangle \lesssim N^{-k} \mper \]
			The same concentration holds for overlaps of a sample with the spike $\bx$.
		\end{lemma}

		We bound the desired cumulants in the following corollary.

		\begin{lemma}
			\label{lem:cumulant-bounds}
			Let $-1 \le s \le 1$ fixed, and $t > 0$ bounded. Let $a$ be a (large) constant. Consider the Ising model $\mu$ with interaction matrix $\beta\bM$, for $\bM = \bW + \frac{\beta}{N} \bx\bx^\top$ (with $\bx \sim \{\pm 1\}^{N}$), and external field
			\[
				\bh = \left( t + \frac{s\beta^2\bs}{N} \right) \cdot \bx + \bg \sqrt{t + \frac{s^2\beta^2}{N}} \mcom
			\]
			with $\bs \sim \{\pm 1\}$ and $\bg \sim \calN\left( 0 , \Id \right)$. Then, for any constants $r\ge 2,a > 0$,
			\[ \E \|\kappa_{r}(\mu)\|_{F}^{2a} \lesssim N^{ra/2} \mper \]
		\end{lemma}
		\begin{proof}
			For $2r$ replicas $\bsigma_1,\ldots,\bsigma_{2r}$ from the Gibbs measure, and denoting $R_{ab} = \frac{1}{N} \langle \bsigma_a , \bsigma_b\rangle$ for shorthand, because $\kappa_r$ is a cumulant, we may express $\|\kappa_{r}(\mu)\|_{F}^2 = N^{r} \cdot \langle P(R_{ab} - q) \rangle$ for some fixed homogeneous degree-$r$ polynomial $P$ depending on $r$ but not $N$.
            \Cref{lem:shifted-model-overlap-conc} then completes the proof, since we may bound
			\[ \E \|\kappa_{r}(\mu)\|_{F}^{2a} = N^{ra} \cdot \E \left\langle P\left( R_{ab} - q \right) \right\rangle^{a} \lesssim N^{ra} \cdot N^{-ra/2} \mcom \]
			where in the final step we use H\"{o}lder's inequality on $P$ to reduce it to bounding $\E \left\langle |R_{12} - q|^{ra} \right\rangle \le N^{-ra/2}$.
		\end{proof}


		The $q = q\left( t + \frac{s^2\beta^2}{N} \right)$ above is defined as the unique solution to the recursion
		\begin{equation}
			\label{eq:planted-overlap-const}
			q(r) = \E_{\bZ \sim \calN(0,1)} \tanh^2 \left( r + \beta^2 q(r) + \bZ \sqrt{ r + \beta^2 q(r) } \right) \mper
		\end{equation}
		In the next section, we will have to bound the above for small $r$.

		The hardest parts of the proof can be black-boxed away to the following lemma due to El Alaoui--Krzkala--Jordan.

		\begin{lemma}[{\cite[Remark after Theorem 8]{EKJ17}}]
			\label{lem:ideal-model-overlap-conc}
			Let $\tau > 0$ bounded. Consider the Ising model $\mu$ with interaction matrix $\beta \bM$ for $\bM = \bW + \frac{\beta}{N} \bx \bx^\top$ and external field
			\begin{equation}
				\label{eq:ideal-ext-field-law-cavity}
				\bh = \tau \bx + \sqrt{\tau} \bg \mper
			\end{equation}
			Then, for some explicit deterministic constant $q = q(\tau)$, we have
			\[ \E \left\langle \left| R_{12} - q \right|^{2k} \right\rangle \lesssim N^{-k} \mper \]
		\end{lemma}

		\begin{remark}
			Although \cite{EKJ17} states their estimate for $\tau$ bounded away from zero, its proof is uniform as $\tau\downarrow0$. The primary ingredients that require verification are (i) $q(\tau) = \Theta(\tau)$ as $\tau \downarrow 0$, (ii) that the preliminary concentration estimate of \cite[Proposition 13]{EKJ17} holds with uniform constants as $\tau \downarrow 0$, and (iii) the matrix $(\Id - t\bA)$ in Eq. (47) is uniformly invertible. We omit the details thereof.
		\end{remark}

		We shall transfer this statement to the model of interest using a contiguity statement, much like how we transitioned to the planted model earlier in \Cref{lem:planted-contiguity}.

		\begin{lemma}
			\label{lem:external-field-contiguity}
			Define by $\nu_{1}$ the law over external fields $\bh_1$ defined by \eqref{eq:real-ext-field-law-cavity} and by $\nu_{2}$ the law over $\bh_2$ defined by \eqref{eq:ideal-ext-field-law-cavity}, for $\tau = t + \frac{s^2\beta^2}{N}$. Then, for any constant $p$ and non-negative measurable function $F$,
			\[ \E_{\nu_1}[F] \lesssim \sqrt{\E_{\nu_2}[F^2]} \mper \]
		\end{lemma}
		\begin{proof}
			It suffices to prove that $\E_{\bh \sim \nu_2} \big[\frac{\dif \nu_1}{\dif \nu_2}(\bh)^{2}\big] \lesssim 1$. Indeed, we then have
			\[ \E_{\nu_1}[F] = \E_{\bh \sim \nu_2} \left[ F(\bh) \cdot \frac{\dif \nu_1}{\dif \nu_2}(\bh) \right] \le \sqrt{ \E_{\nu_2}[F^2] } \cdot \sqrt{ \E_{\bh \sim \nu_2}\left[ \left(\frac{\dif \nu_1}{\dif \nu_2}\right)^{2} \right] } \]
			by the Cauchy--Schwarz inequality. To prove this, we can establish this bound after conditioning on $\bx$ and $\bs$. Assume we fix $\bs = +1$ (the proof when $\bs = -1$ is identical), so we have
			\begin{align*}
				\bh_1 &\sim \calN\left( \bx \cdot \left(t + \frac{s\beta^2}{N}\right) , t + \frac{s^2\beta^2}{N} \right) \text{ and} \\
				\bh_2 &\sim \calN\left( \bx \cdot \left(t + \frac{s^2\beta^2}{N}\right) , t + \frac{s^2\beta^2}{N} \right) \mper
			\end{align*}
			A straightforward calculation on these Gaussians, then shows that
			\begin{align*}
            \E_{\bh \sim \nu_2} \left[\frac{\dif \nu_1}{\dif \nu_2}(\bh)^{2}\right] &\lesssim \exp\left(\frac{1}{t + \frac{s^2\beta^2}{N}} \cdot \left\| \bx \cdot \left(\left(t + \frac{s\beta^2}{N}\right) - \left(t + \frac{s^2\beta^2}{N}\right) \right) \right\|^2\right)\\ &\lesssim \exp\left(\frac{4N}{s^2\beta^2} \cdot \frac{\|\bx\|^2 \cdot s^2\beta^4}{N^2}\right)= e^{4\beta^2} \mper
            \end{align*}
			The lemma follows.
		\end{proof}

		We now complete the proof as follows.

		\begin{proof}[Proof of \Cref{lem:shifted-model-overlap-conc}]
			Let $\mu_1$ be the (random) model defined in \Cref{lem:shifted-model-overlap-conc}, and $\mu_2$ the (random) model defined in \Cref{lem:ideal-model-overlap-conc}. We then have by \Cref{lem:external-field-contiguity} that
			\begin{align*}
				\E_{\nu_1} \left\langle \left| R_{12} - q \right|^{2k} \right\rangle &\lesssim \sqrt{ \left(\E_{\nu_2} \left\langle \left| R_{12} - q \right|^{2k} \right\rangle^{2} \right)  } \le \sqrt{ \left(\E_{\nu_2} \left\langle \left| R_{12} - q \right|^{4k} \right\rangle \right) } 
					\lesssim \frac{1}{N^{k}} \mcom
			\end{align*}
			as desired, where the middle inequality is Jensen's. For the overlap concentration with a spike, we note that by exchangeability and the posterior structure of the distribution, \Cref{lem:ideal-model-overlap-conc} holds if one of the samples is replaced with the spike.
		\end{proof}


\section{Lower bounds for the prediction of the inverse covariance}
\label{sec:lower-bd-inv-cov-pred}

	Recall the definition of the matrix $K$ from \eqref{eq:inv-cov-pred},
	\[ K = \diag\left( \frac{1}{1 - m_i^2} \right) - \beta \bM + \beta^2(1-q_N)\Id - \frac{2\beta^2}{N} mm^{\top} \]
	The goal of this subsection will be to prove the following more precise version of \Cref{lem:organized-gap}.

	\begin{lemma}
		\label{lem:cov-pred-bound}
		Let $T > 0$ large. There exists a constant $\gamma = \gamma_{\beta,T}$ such that for any constant $a$ and $t \in [0,T]$,
		\begin{equation}
			\label{eq:cov-pred-bound}
			\Pr\left[ K_{t} \not\psdge \gamma \right] \le N^{-a} \mper
		\end{equation}
	\end{lemma}

	\subsection{The case when $t$ is small}

		We start with the regime where $t$ is small.

		\begin{lemma}[{\cite[Corollary 1.2]{Maida07}}]
			\label{lem:cov-pred-bound-zero}
			\eqref{eq:cov-pred-bound} holds for $t = 0$.
		\end{lemma}
		
		\begin{lemma}
			\label{lem:cov-pred-bound-small}
			There is some small constant $t_0 = t_0(\beta)$ such that \eqref{eq:cov-pred-bound} holds for all $0 \le t < t_0$.
		\end{lemma}
		\begin{proof}
			We have
			\begin{align*}
						K_{t} - K_{0} &= \left(\diag\left( \frac{1}{1-m_i^2} \right) - \Id\right) - \beta^2 q_{N} \Id - \frac{2\beta^2}{N} mm^\top \\
							&\psdge - 3 \beta^2 q_{N} \mper
			\end{align*}
			It suffices to show that for any $\eps > 0$, for sufficiently small $t$, $q_{N} = \frac{\|m\|^2}{N} \le \eps$ with high probability.
			This would allow us to choose $\eps = \frac{\gamma_0}{12\beta^2}$, for $\gamma_0$ the lower bound yielded by \Cref{lem:cov-pred-bound-zero}. 
			We have
			\[ \frac{\|m\|^2}{N} = \langle R_{12}\rangle = q + \langle R_{12} - q \rangle \mper \]
			By the overlap concentration of \Cref{lem:shifted-model-overlap-conc}, we have that with probability $1 - N^{-a}$, $\langle R_{12} - q\rangle \le \eps$.
			It remains to establish that $q = q(t)$ is bounded, for $q(t)$ defined as in \eqref{eq:planted-overlap-const}. Consider $f$ defined by $f(r) = \E \tanh^2( r + \sqrt{r} \bZ )$, so $q(t) = f\left( t + \beta^2 q(t) \right)$. We may show that $f'(r) = \E \sech^4\left( r + \sqrt{r} \bZ \right) \le 1$, so $q(t) \le t + \beta^2 q(t)$. Rearranging proves the desideratum for $t$ sufficiently small.
		\end{proof}

	\subsection{The case when $t$ is large}

		The bulk of the proof requires understanding the mean when $t$ is bounded away from $0$.
		
		\begin{restatable}{lemma}{positivetpredictionlb}
			\label{lem:cov-pred-bound-large}
			For $t$ bounded away from $0$, \eqref{eq:cov-pred-bound} holds.
		\end{restatable}

		Fortunately, much of the hard work may be black-boxed away to pre-existing papers by Celentano \cite{Cel24} and El Alaoui--Montanari--Sellke \cite{EMS22}. Let us do some setup before describing exactly what these results are.

		Recall
		\[ K(m) = \diag\left( \frac{1}{1 - m_i^2} \right) - \beta \bM + \beta^2(1-q_N)\Id - \frac{2\beta^2}{N} mm^{\top} \mper  \]

		Throughout this section, let $\eps > 0$ be a small constant and $t$ bounded away from $0$. The goal of this section will be to prove the following.

		\begin{restatable}{lemma}{midprobpredbound}
			\label{lem:mid-probability-pred-bound}
			Let $\eps > 0$ a small constant, $a > 0$ a large constant, and $t$ be bounded away from $0$. There exists an approximation $\wh{m} = \wh{m}(\bW,\bg)$ of the mean of $\mu_t$ such that
			\begin{enumerate}[label=(\alph*)]
				\item $\Pr\left[\|\wh{m} - m\|_{2} \ge \eps \sqrt{N}\right] \le N^{-a}$,
				\item for some constant $\gamma$, with probability $1-o(1)$, for all $u\in(-1,1)^N$ with $\|\wh{m} - u\| \le \eps\sqrt{N}$, $K(u) \psdge \gamma$, and
				\item $\wh{m}$ is a $O_{\eps}(\sqrt{N})$-Lipschitz function of $W$ and an $O(1)$-Lipschitz function of $g$ with high probability. More quantitatively, if $\|W\|,\|W'\|$ are bounded by some constant, we have $\|\wh{m}(W,g) - \wh{m}(W',g')\| \le L\sqrt{N}\|W - W'\|  + L \|g - g'\|$ for some constant $L$ depending on $\eps$ but not $N$.
			\end{enumerate}
		\end{restatable}

		Before proving the above, let us describe how to complete the proof---this requires us to boost the probability in (b) from $1-o(1)$ to $1 - N^{-a}$. We will do this using Gaussian concentration.

		\positivetpredictionlb*

		\begin{proof}
			Define the set of ``good'' disorder $\Egood$ by
			\[ \Egood = \left\{ (W,g) : \text{for all $u$ with $\|\wh{m}(W,g) - u\| \le 2 \eps \sqrt{N}$, $K_{W,g}(u) \psdge 2\gamma$} \right\} \mcom \]
			so $\Pr\left[ \Egood \right] \ge 1/2$ by \Cref{lem:mid-probability-pred-bound}(b). First, observe that by definition,
			\begin{equation}
				\label{eq:silly-bound-on-K-lipschitzness}
				\|K_{W,g}(u) - K_{W',g'}(u) \|_2 \le \| W - W' \|_2 \mper
			\end{equation}
			Suppose that we want to establish that $K_{W,g}(m)$ is bounded away from $0$. Further suppose that we had, for some $(W',g') \in \Egood$, that
			\begin{enumerate}[label=(\roman*)]
				\item $\|W - W'\|_{2} + \frac{1}{\sqrt{N}} \|g - g'\|_2 \le \gamma$,
				\item $\|m(W,g) - \wh{m}(W,g)\| \le \eps\sqrt{N}$, and
				\item $\|\wh{m}(W,g) - \wh{m}(W',g')\| \le \eps\sqrt{N}$.
			\end{enumerate}
			This would complete the proof. Indeed, (ii) and (iii) imply that $\|m(W,g) - \wh{m}(W',g')\| \le 2\eps\sqrt{N}$, so the definition of $\Egood$ gives that $K_{W',g'}(m(W,g)) \psdge 2\gamma$. \eqref{eq:silly-bound-on-K-lipschitzness} then implies that $K_{W,g}(m(W,g)) \psdge K_{W',g'}(m(W,g)) - \| W - W' \|_{2} - \frac{1}{\sqrt{N}} \|g - g'\|_2 \psdge \gamma$, as desired.

			It remains to show that the three conditions hold with high probability.
			(ii) holds with probability $1 - N^{-a}$ by \Cref{lem:mid-probability-pred-bound}(a), and the Lipschitzness of \Cref{lem:mid-probability-pred-bound}(c) reduces (iii) to a statement like (i). To conclude, it suffices to establish that for any small constant $\eta$, with high probability over $\bW$, there exists $(W',g') \in \Egood$ such that $\|W - W'\| + \frac{1}{\sqrt{N}} \|g-g'\| \le \eta$. This in fact occurs with probability $1 - e^{-cN}$ using Gaussian isoperimetry, e.g. \cite[Theorem 10.15]{BLM13}---this implies that for any set $S$ of standard Gaussian measure at least $1/2$, $\Pr_{\bg \sim \calN\left(0,1\right)}\left[ \inf_{x \in S} \|\bg - x\| \ge r \right] \le e^{-r^2/2}$, completing the proof.
		\end{proof}

	\subsection{The proof of \Cref{lem:mid-probability-pred-bound}}

		In this section, we prove the primary technical input to the covariance bound when $t$ is bounded away from $0$.

		\midprobpredbound*

		Towards transparency, let us explain how the approximation $\wh{m}$ is defined---it will not be too important to understand the details of all the terms here. Set $\bh = t\bx + \bB_{t}$ for the external field. Initialize
		\[ \wh{m}^{(-1)} = 0 \text{ and } \wh{m}^{(0)} = \tanh(\bh) \mper \]
		For $k \ge 0$, define
		\[ \wh{m}^{(k+1)} = \tanh\left( \beta \bM \wh{m}^{(k)} + \bh - \beta^2 \left( 1 - \frac{\| \wh{m}^{(k)} \|^2}{N} \right) \wh{m}^{(k-1)} \right) \mcom \]
		where $\tanh$ is applied entrywise.
		We will choose $\wh{m} = \wh{m}^{(k_*)}$ for some large constant $k_*$ depending on $\eps$ but not $N$.

		\begin{proof}[Proof of \Cref{lem:mid-probability-pred-bound}(a,c)]
			(c) is immediate from the definition of $\wh{m}$, since we are choosing a large constant iterate of this process. For (a), we wish to show that $m^\top \wh{m}$ concentrates. Let $\calD = (\bM , \bh)$, so $m = \E[ \bx \mid \calD ]$. That is,
			\[ \frac{m^\top \wh{m}}{N} = \E\left[ \frac{\bx^\top \wh{m}}{N} \mid \calD \right] \mper \]
			For any fixed iterate $k$, \cite[Proposition 3]{Cel24} establishes that $\frac{\bx^\top \wh{m}}{N} \to q_{k}$ in probability, for some explicit constant $q_{k}$. Consider the random variable $\rho = \frac{\bx^\top \wh{m}}{N}$, conditioned on $\bx$, defined in terms of the disorder $\bW$ and the noise in the external field $\bg$. Using the Kirszbraun extension theorem and the Lipschitzness of (c), we may define a function $\wt{m}$ that is $O(\sqrt{N})$-Lipschitz in $W$ everywhere, and $\wt{m} = \wh{m}$ with probability $1 - e^{-cN}$. Define $\wt{\rho} = \frac{\bx^\top \wt{m}}{N}$. Gaussian concentration yields that
			\[ \Pr\left[ | \wt{\rho} - \E \wt{\rho} | > s \right] \le e^{-cNs^2} \mcom \]
			and $\rho \to q_{k}$ in probability also implies that for sufficiently large $k$, for constant $\delta$,
			\[ \Pr\left[ | \rho - q_k | > \delta \right] \le e^{-cN} \mper \]
			Therefore,
			\[ \Pr\left[ \left| \frac{m^\top \wh{m}}{N} - q_{k} \right| \ge \delta \right] \le e^{-cN} \]
			To conclude, we write
			\[ \frac{\| m - \wh{m} \|^2}{N} = \frac{\|m\|^2}{N} + \frac{\|\wh{m}\|^2}{N} - 2 \cdot \frac{\wh{m}^\top m}{N} \mper \]
			Overlap concentration (\Cref{lem:shifted-model-overlap-conc}) implies that $\left| \frac{\|m\|^2}{N} - q \right| \le \delta$ with probability $1 - N^{-a}$, and we have just shown that $\left| \frac{\wh{m}^\top m}{N} - q_{k} \right| \le \delta$ with probability $1 - N^{-a}$. An argument very similar to the Lipschitzness one above also yields that $\left| \frac{\|\wh{m}\|^2}{N} - q_{k} \right| \le \delta$, using \cite[Proposition 3]{Cel24}, with probability $1 - N^{-a}$. Taking $k$ large enough that $|q_k - q| \le \delta$, and adjusting the constants establishes (a).
		\end{proof}

		The remainder of this section will be spent on (b). In fact, \cite[Theorem 2(i)]{Cel24} establishes that $K(\wh{m}) + \frac{2\beta^2}{N} \wh{m} \wh{m}^\top \psdge \gamma$ with high probability, and our proof will be spent showing that the extra rank-one term is innocuous.

		The proof will follow the same blueprint as \cite{Cel24}, but there will be an extra term (corresponding to the rank-one term) that will slightly influence the final step of the proof.

		\begin{proof}[Proof of \Cref{lem:mid-probability-pred-bound}(b)]
			Throughout this proof, we denote by $\delta_{k}$ a small constant which goes to $0$ as the iterate $k_*$ defining $\wh{m}$ goes to $\infty$.
			First off, the fact that $\|u - \wh{m}\|$ is small means that we only need to argue about the ``mixed'' matrix,
			\[ \wt{K}(u) \defeq \diag\left( \frac{1}{1 - u_i^2} \right) - \beta \bM + \beta^2\left( 1 - \frac{\|u\|^2}{N} \right)\Id - \frac{2\beta^2}{N} \wh{m}\wh{m}^{\top} \mper \]
			Let $v \in \bbS^{N-1}$ be an arbitrary vector (possibly depending on the disorder in some uncontrolled way), so we are interested in showing the (strict) negativity of
			\[ - v^\top \wt{K}(u) v = \beta v^\top \bW v + R(u,v) \mcom \]
			for some function $R = R_{\bW}$. \cite[Corollary 1]{Cel24}, the ``post-AMP Sudakov--Fernique inequality,'' establishes that the first term may be replaced as
			\[ \operatorname{plimsup}_{N \to \infty} \sup_{u,v} \left\{ \beta v^\top \bW v + R(u,v) \right\} \le \operatorname{plimsup}_{N \to \infty} \sup_{u,v} \left\{ \frac{2\beta}{\sqrt{N}} \left\langle g_{\mathrm{AMP}}(v) , v \right\rangle + R(u,v)\right\} \]
			for some explicit function $g_{\mathrm{AMP}}$ defined by the disorder, whose definition we will refrain from introducing. Let $\bI \sim [N]$ be a random index, and define $M = \wh{m}_{\bI}$, $U = u_{\bI}$, and $V = \sqrt{N} v_{\bI}$. The extra rank-one term $\frac{2\beta^2}{N} \langle v , \wh{m} \rangle^2$ is then precisely $2\beta^2 \E[MV]^2$. 
			Assume wlog that $\bx = \bone$.

			The next step is to move to appropriate \emph{limiting} distributions for all these random variables. At the cost of some small error $\delta_{k}$ that goes to $0$ as the iterate $k$ (defining $\wh{m}$), this allows us to work with idealized distributions that are explicitly defined in terms of Gaussians. For example, we can pretend that $M$ is distributed as $\tanh\left( t+\beta^2q + \bz \sqrt{t+\beta^2 q} \right)$ for $\bz \sim \calN(0,1)$, up to some small error.

			The upshot, which is the content of \cite[Sections 4.2--4.4]{Cel24}, is essentially \cite[Eq. (26)]{Cel24} with an extra term corresponding to the rank-one part. This is as follows.
			\[ \calJ(U,V) = 2\beta\E[B_{\rho} V] - \E\left[ \frac{V^2}{1 - U^2} \right] + \beta^2 (\E[V])^2 - \beta^2 (1 - q) + \underbrace{2\beta^2 q \rho^2}_{\text{new term}} \mcom \]
			where we have the constraints
			\[ \E[V^2] = 1 \mcom \qquad |U| \le 1 \mcom \qquad \E[SV] = \rho\sqrt{q} \mcom \qquad \|U-S\|_{2} \le \eps + \delta_{k} \]
			and for some choice of $\bz_1,\bz_2,\bz_3 \sim \calN(0,1)$ that are independent from each other,
			\begin{equation}
			\begin{aligned}
				S &= \tanh\left( t+\beta^2q + \bz_1 \beta\sqrt{q} + \bz_2 \sqrt{t} \right) \mcom \\
				B_{\rho} &= \rho \bz_1 +  \bz_3 \sqrt{1 - \rho^2} \mper
			\end{aligned}
			\label{eq:s-brho-definitions}
			\end{equation}
			In particular, $\E[S^2] = q$, and $\rho = \frac{\E[SV]}{\sqrt{q}} \in [-1,1]$. Here, $\bS$ plays the role of the typical coordinate of $\wh{m}$. The next step is to also replace the $\bU$ with $\bS$, suffering a small constant loss.

			Towards this, we have for any choice of $a,b,\alpha$ that
			\begin{align*}
				\calJ &= 2\beta\E[B_{\rho} V] - \E\left[ \frac{V^2}{1 - U^2} \right] + \beta^2 (\E[V])^2 - \beta^2 (1 - q) + 2\beta^2 q \rho^2 \\
				&\qquad\qquad + a \left( \frac{\E[SV]}{\sqrt{q}} - \rho \right) + b\left( \E[V] - \E[V] \right) + \alpha\left( \E[V^2] - 1 \right) \\
					&= \E\left[ V \left( 2\beta B_{\rho} + \frac{aS}{\sqrt{q}} + b \right) - V^2 \left(\frac{1}{1 - U^2} -  \alpha \right) \right] \\
					&\qquad\qquad + \beta^2 (\E[V])^2 - \beta^2 (1 - q) + 2\beta^2 q \rho^2 - a\rho - b\E[V] - \alpha \\
					&\le \frac{1}{4} \E\left[ \frac{ \left( 2\beta B_{\rho} + \frac{aS}{\sqrt{q}} + b \right)^2 }{ \frac{1}{1-U^2} - \alpha } \right] + \beta^2 (\E[V])^2 - \beta^2 (1 - q) + 2\beta^2 q \rho^2 - a\rho - b\E[V] - \alpha \mper
			\end{align*}
			The final inequality here is by the AM-GM inequality: we have $cv - dv^2 \le \frac{c^2}{4d}$. Note now that if $0 \le \alpha \le 1/2$, we may replace $\bU$ with $\bS$ in the first expression, at the cost of a small loss. After doing this, we are left with an extremely explicit calculation in terms of Gaussians, since essentially all dependence on $\bV$ has been removed. Now denote $\eta = \E[V]$.

			The result is a modification \cite[Eq. (27)]{Cel24}, which asks to show that
			\[ \inf_{0 \le \alpha \le 1/2} \calL_\alpha \defeq \inf_{\substack{a,b \in \R \\ 0 \le \alpha \le 1/2}} \frac{1}{4} \E\left[ \frac{ \left( 2\beta B_{\rho} + \frac{aS}{\sqrt{q}} + b \right)^2 }{ \frac{1}{1-S^2} - \alpha } \right] + \beta^2 \eta^2 - \beta^2 (1 - q) + 2\beta^2 q \rho^2 - a\rho - b\eta - \alpha \]
			is negatively bounded away from $0$. In fact, it will suffice to show that $\calL_0 \le 0$---the remainder of the proof follows \cite[Section 4.4]{Cel24}, which uses a simple continuity argument to transform this to a strict negative margin by setting $\alpha$ to be a small positive constant.
			To do this, we set $a = 2\beta^2(2q\rho - \eta\sqrt{q})$ and $b = 2\beta^2\eta$. This then simplifies to
			\[ \calL_0 \le - \beta^2 (1 - \beta^2 \nu) \left( q\rho^2 + (\eta - \rho\sqrt{q})^2 \right) \mcom \]
			for $\nu = \E[1-S^2]$. We omit the details of this calculation---it amounts to integrating out $\bz_3$, using Gaussian integration by parts to explicitly compute terms like $\E[(1-S^2) \bz_1^2]$ (which show up when expanding out the parenthesised terms in the expectation), and completing appropriate squares---such calculations are very similar to those done in \cite[pages 16 and 17]{Cel24}.
		\end{proof}

\section*{Acknowledgments}
We would like to thank David X.\ Wu and Mark Sellke for many technical discussions during earlier stages of this project.
S.M.\ and A.R.\ would like to thank the Simons Institute for the Theory of Computing at Berkeley for their hospitability and hosting them at the ``Spectral Theory Beyond Graphs'' program, where part of the work was conducted.
H.D.\ was partially supported by an NSF-Simons research collaboration grant (award number 2031883).
B.H.\ was supported by a Stanford Science Fellowship and NSF Mathematical Sciences Postdoctoral Fellowship.

\addcontentsline{toc}{section}{References}

\bibliographystyle{alpha}
\bibliography{main}

\crefalias{section}{appendix}

\appendix

\section{Tying up loose ends}   \label{app:technical-estimates}

We start by proving \Cref{lem:large-t-poinc}, which is essentially an immediate consequence of the techniques of \cite{BER26}.

\largetpoinc*

\begin{proof}
    Assume without loss of generality that $\bx = \bone$. We start by noting that \cite[Lemma 2.3]{BER26} holds for sufficiently large localization time. Indeed, choose $\delta$ such that the first condition of \cite[Theorem 1.1]{BER26} is satisfied, and every small submatrix has sufficiently small operator norm, which occurs with probability $1 - e^{-cN}$. The external field is of the form $T\bx + \bg\sqrt{T}$, For sufficiently large $T$, with probability $1 - e^{-cN}$, we have $\|\bg\sqrt{T}\| \ll T\sqrt{N}$. Consequently, a sample from this measure has at least $\left( 1 - \frac{\delta}{2} \right)N$ positive spins, which by \cite[Lemma 2.2]{BER26} yields a correlation bound (for the base model as well as the pinned models).

    To conclude, we use this correlation bound until we have pinned $(1-\delta)N$ of the coordinates, and we are left with an Ising model (with some external field) whose interaction matrix has bounded spectral radius; \cite[Corollary 51]{CE22} yields a \Poincare inequality for these measures, and the spectral independence framework allows us to lift this back to a \Poincare inequality for $\mu_{T}$.
\end{proof}

\subsection{Controlling error terms}

Over the remainder of this section, we prove \Cref{lem:error-bounds,lem:organized-taylor-errors}.
For $s\in[-1,1]$, let $\kappa_{r,s}^{(i)}$ denote $\grad^r F^{(i)}(u_i+s\beta\bM_{i,-i})$---in words, the $r$th cumulant of the model where the $i$th site is removed, and the field is shifted by $s\beta\bM_{i,-i}$.
More generally for $B\subseteq[N]$, we will use $\kappa_{r,s}^{(B)}$ to denote $\grad^r F^{(i)}\parens*{u_B + s \beta \sum_{i\in B} \bM_{i,-i} }$ where $u_B$ refers to the external field on the remaining sites.

We will rely on the following estimates for the cumulant, whose proof we defer to \Cref{subsec:proof-of-cumulant-contraction}.


\begin{lemma}   \label{lem:cumulant-contraction-estimates}
    For any fixed integers $r\ge3$ and $k,\ell,p\ge0$ satisfying $2k+\ell\le r$, and fixed $s\in[-1,1]$,
    \[
        \E\left\|
            \kappa_{r,s}^{(i)}
            \left[
                \Id^{\otimes k}\otimes
                \bM_{i,-i}^{\otimes\ell}
            \right]
        \right\|_F^p
        \le c_p N^{p(r/4-\ell/2)}.
    \]
    Further, the same bound continues to hold if $\kappa_{r,s}^{(i)}$ is replaced with $\kappa_{r,s}^{(B)}$ where $|B|$ is a constant independent of $N$.
\end{lemma}

We are now ready to prove \Cref{lem:organized-taylor-errors}, restated for convenience.

\errorboundstwo*

\begin{proof}
    By Taylor's theorem, the remainder term admits the following description:
    \begin{align*}
        \Delta_{d,i}
        &=
        \frac{\beta^5}{48}\int_0^1(1-s)^4
        \bigl(\kappa_{6,s}^{(i)}+\kappa_{6,-s}^{(i)}\bigr)
        [\bM_{i,-i}^{\otimes5}]\,ds,\\
        \Delta_{\Cov,i}\bM_{i,-i}
        &=
        \frac{\beta^3m_i}{4}\int_0^1(1-s)^2
        \bigl(\kappa_{5,s}^{(i)}+\kappa_{5,-s}^{(i)}\bigr)
        [\bM_{i,-i}^{\otimes4}]\,ds\\
        &\quad+
        \frac{\beta^4}{12}\int_0^1(1-s)^3
        \bigl(\kappa_{6,s}^{(i)}+\kappa_{6,-s}^{(i)}\bigr)
        [\bM_{i,-i}^{\otimes5}]\,ds\mper
    \end{align*}
    By \Cref{lem:cumulant-contraction-estimates} and triangle inequality, we have:
    \[
        \E\norm*{ \Delta_{d,i} }_F^p \le C N^{-p},\qquad \E\norm*{ \Delta_{\Cov,i} \bM_{i,-i} }_F^p \le C N^{-3p/4}\mper
    \]
    The conclusion follows by applying Markov's inequality to the above for an appropriately large choice of $p$ depending on $\eps$.
\end{proof}

We now turn our attention to proving \Cref{lem:error-bounds}.
A key ingredient is the following set of improved cumulant bounds for particular special cases.
We defer the proof to \Cref{subsec:proof-of-improved-cumulant}.
\begin{lemma}[Improved cumulant bounds]
    \label{lem:improved-cumulant}
    There exists an absolute constant $C>0$ such that the following holds.
    \begin{enumerate}[label=(\alph*)]
        \item \label{it:improved-cumulant-kappa34}
        For all $p\ge 2$, there exists $C_p > 0$ such that
        \[
            \max\left\{
                \E \|\kappa_3\|_F^p,
                \max_{i\in [N]} \E \|\kappa_3^{(i)}\|_F^p
            \right\} \le C_p N^{p/2}\mcom \qquad
            \max\left\{
                \E \|\kappa_4\|_F^p,
                \max_{i\in [N]} \E \|\kappa_4^{(i)}\|_F^p
            \right\} \le C_p N^{3p/4}\mper
        \]
        \item \label{it:improved-cumulant-cov}
        For all $p\ge 2$, there exists $C_p > 0$ such that
        \[
            \max\left\{
                \E \opnorm{\Cov}^p,
                \max_{i\in [N]} \E \opnorm{\Cov^{(i)}}^p
            \right\} \le C_p \mper
        \]        
        \item \label{it:improved-cumulant-cov-high-prob}
        For all $a>0$, there exists $C_a > 0$ such that
        \[
            \Pr\left[
                \max\left\{
                    \|\Cov\|_F,\,
                    \max_{i\in[N]}\|\Cov^{(i)}\|_F
                \right\}
                > C\sqrt N
            \right]
            \le C_a N^{-a} \mper
        \]
        \item \label{it:improved-cumulant-kappa3-high-prob}
        For all $a>0$, there exists $C_a > 0$ such that
        \[
            \Pr\left[
                \max\left\{
                    \|\kappa_3\|_F,\,
                    \max_{i\in[N]}\|\kappa_3^{(i)}\|_F
                \right\}
                > C\sqrt N
            \right]
            \le C_a N^{-a} \mper
        \]
    \end{enumerate}
\end{lemma}
We now prove \Cref{lem:error-bounds}.

\subsection{Proof of \Cref{lem:error-bounds}}

\errorboundsone*

\textbf{Bound on $\left\|{\frac{2\beta^2 m_i}{N} \parens*{ \Cov^{(i)}\cdot m^{(i)} - (\Cov\cdot m)_{-i} } }\right\|$.}
First, observe that:
\[
    \grad_{u_i} \bracks*{\Tr\Cov - \Tr\Cov^{(i)}} = {2} \left( \Cov^{(i)}m^{(i)} - (\Cov\cdot m)_{-i} \right)\mper
\]
By \eqref{eq:organized-blocks}, we have:
\begin{align}
    \Tr\Cov - \Tr\Cov^{(i)} &= \frac{1+m_i}{2} \Tr\kappa_{2,+}^{(i)} + \frac{1-m_i}{2} \Tr\kappa_{2,-}^{(i)} - \Tr\kappa_{2}^{(i)} + v_i\parens*{ 1+ \|d_i\|_2^2 }\nonumber \\
    &= \frac{\Tr\parens*{ \kappa_{2,+}^{(i)} + \kappa_{2,-}^{(i)} }}{2} + m_i \frac{\Tr\parens*{ \kappa_{2,+}^{(i)} - \kappa_{2,-}^{(i)} }}{2} - \Tr\kappa_2^{(i)} + v_i(1+\|d_i\|_2^2)
    \mper\label{eq:cov-decomposition}
\end{align}
In preparation to compute the gradient of the above, observe that:
\[
    \nabla_{u_i}m_i=v_id_i,
    \qquad
    \nabla_{u_i}v_i=-2m_iv_id_i,
    \qquad
    \nabla_{u_i}d_i
       =\frac{\kappa_{2,+}^{(i)}-\kappa_{2,-}^{(i)}}2,
    \qquad
    \nabla_{u_i}\operatorname{Tr}\kappa_{2,\pm}^{(i)}
    =\kappa_{3,\pm}^{(i)}[\Id]
\]
We differentiate the three terms in
\[
    \operatorname{Tr}\Cov-\operatorname{Tr}\Cov^{(i)}
    =
    \left(
        \frac{\operatorname{Tr}
            (\kappa_{2,+}^{(i)}+\kappa_{2,-}^{(i)})}{2}
        -\operatorname{Tr}\kappa_2^{(i)}
    \right)
    +\frac{m_i}{2}\operatorname{Tr}
        (\kappa_{2,+}^{(i)}-\kappa_{2,-}^{(i)})
    +v_i(1+\|d_i\|_2^2)
\]
separately. First,
\[
    \nabla_{u_i}\left(
        \frac{\operatorname{Tr}
            (\kappa_{2,+}^{(i)}+\kappa_{2,-}^{(i)})}{2}
        -\operatorname{Tr}\kappa_2^{(i)}
    \right)
    =
    \frac12
        \bigl(\kappa_{3,+}^{(i)}+\kappa_{3,-}^{(i)}
              -2\kappa_3^{(i)}\bigr)[\Id].
\]
Next, by the product rule,
\begin{align*}
    \nabla_{u_i}\left(
        \frac{m_i}{2}\operatorname{Tr}
            (\kappa_{2,+}^{(i)}-\kappa_{2,-}^{(i)})
    \right)
    ={}&
    \frac{v_i}{2}
        \operatorname{Tr}
            (\kappa_{2,+}^{(i)}-\kappa_{2,-}^{(i)})\,d_i
        +\frac{m_i}{2}
        \bigl(\kappa_{3,+}^{(i)}-\kappa_{3,-}^{(i)}\bigr)[\Id].
\end{align*}
Finally,
\begin{align*}
    \nabla_{u_i}\bigl(v_i(1+\|d_i\|_2^2)\bigr)
    &=(1+\|d_i\|_2^2)\nabla_{u_i}v_i
        +2v_i(\nabla_{u_i}d_i)^\top d_i\\
    &=-2m_iv_i(1+\|d_i\|_2^2)d_i
        +v_i\bigl(\kappa_{2,+}^{(i)}
                    -\kappa_{2,-}^{(i)}\bigr)d_i\mper
\end{align*}
From the above, we obtain:
\begin{align*}
    \left(
        \Cov^{(i)}m^{(i)}-(\Cov m)_{-i}
    \right)
    &=
    \frac1{4}
        \bigl(\kappa_{3,+}^{(i)}+\kappa_{3,-}^{(i)}
              -2\kappa_3^{(i)}\bigr)[\Id]
        +\frac{m_i}{4}
        \bigl(\kappa_{3,+}^{(i)}
              -\kappa_{3,-}^{(i)}\bigr)[\Id]\\
    &+\frac{v_i}{4}
        \operatorname{Tr}
            (\kappa_{2,+}^{(i)}-\kappa_{2,-}^{(i)})\,d_i
        - m_iv_i(1+\|d_i\|_2^2)d_i
    +\frac{v_i}{2}
        \bigl(\kappa_{2,+}^{(i)}
              -\kappa_{2,-}^{(i)}\bigr)d_i.
\end{align*}
We bound the five terms on the right-hand side in order.

\smallskip

\noindent \emph{The first term.}
By Taylor's formula, we have:
\begin{align*}
    &\frac1{4}
      \bigl(\kappa_{3,+}^{(i)}+\kappa_{3,-}^{(i)}
             -2\kappa_3^{(i)}\bigr)[\Id]=
    \frac{\beta^2}{4}\int_0^1(1-s)
        \bigl(\kappa_{5,s}^{(i)}+\kappa_{5,-s}^{(i)}\bigr)
        [\Id\otimes\bM_{i,-i}^{\otimes2}]\,ds\mper
\end{align*}
By \Cref{lem:cumulant-contraction-estimates}, the above is at most $O\parens*{N^{1/4+\eps}}$ with probability $1-O(N^{-a})$.

\smallskip

\noindent \emph{The second term.}
By Taylor's theorem:
\begin{align*}
    &\frac{m_i}{4}
        \bigl(\kappa_{3,+}^{(i)}-\kappa_{3,-}^{(i)}\bigr)[\Id]=
    \frac{\beta m_i}{2}
        \kappa_4^{(i)}[\Id\otimes\bM_{i,-i}]+
    \frac{\beta^3m_i}{8}\int_0^1(1-s)^2
        \bigl(\kappa_{6,s}^{(i)}+\kappa_{6,-s}^{(i)}\bigr)
        [\Id\otimes\bM_{i,-i}^{\otimes3}]\,ds.
\end{align*}
By \Cref{lem:cumulant-contraction-estimates}, the second term is $O\parens*{N^{\eps}}$ with probability $1-O(N^{-a})$.

To control the first term, observe that we may write:
\[
    \kappa_4^{(i)}[\Id\otimes \bM_{i,-i}] = \kappa_4^{(i)}[\Id] \bM_{i,-i}\mper
\]
Since $\kappa_4^{(i)}[\Id]$ is also the second derivative of $\Tr\Cov^{(i)}$, we see that it is equal to $-2(\Cov^{(i)})^2 - 2\kappa_3^{(i)}[m^{(i)}]$.
First, notice that $\norm*{(\Cov^{(i)})^2}_F \le \sqrt{N}\norm*{\Cov^{(i)}}_{\mathrm{op}}^2$.
By \Cref{lem:improved-cumulant}, we have $\E\left\|{(\Cov^{(i)})^2}\right\|_F^{p} \le c_p N^{p/2}$.

For every symmetric matrix $B$ with $\|B\|_F=1$, we have:
\begin{align*}
    \bigl|\langle\kappa_3^{(i)}[m^{(i)}],B\rangle\bigr|^2
    &\le
    \left((m^{(i)})^\top\Cov^{(i)}m^{(i)}\right) \cdot
       \Var_{\bsigma \sim \mu_{-i}}\left[ (\bsigma-m^{(i)})^\top B(\bsigma-m^{(i)}) \right]\\
    &=
    \bigl((m^{(i)})^\top\Cov^{(i)}m^{(i)}\bigr)
    \left(
        \kappa_4^{(i)}[B\otimes B]
        +2\operatorname{Tr}(\Cov^{(i)}B\Cov^{(i)}B)
    \right)\\
    &\le
    N\|\Cov^{(i)}\|_{\mathrm{op}}
    \left(
        \|\kappa_4^{(i)}\|_F
        +2\|\Cov^{(i)}\|_{\mathrm{op}}^2
    \right)\mcom    \numberthis \label{eq:third-cumulant-bound}
\end{align*}
where both inequalities are by Cauchy-Schwarz, and the fact that $\norm*{m_i}^2 \le N$.
By choosing the extremal $B$ and applying
\Cref{lem:improved-cumulant}, we obtain
\[
    \E\left\|\kappa_3^{(i)}[m^{(i)}]\right\|_F^p=O(N^{7p/8})\mper
\]
Also,
$\|(\Cov^{(i)})^2\|_F
\le\sqrt N\|\Cov^{(i)}\|_{\mathrm{op}}^2$.
Consequently,
\[
    \E\left\|\kappa_4^{(i)}[\Id]\right\|_F^p \le O\parens*{N^{7p/8}}\mper
\]
Returning to the goal of obtaining a bound on $\left\|{\kappa_4^{(i)}[\Id]\bM_{i,-i}}\right\|$, we get:
\begin{align*}
    \left\|\kappa_4^{(i)}[\Id]\bM_{i,-i}\right\|_2 &\le \left\|\kappa_4^{(i)}[\Id]\bW_{i,-i}\right\|_2 + \frac{1}{N}\left\|\kappa_4^{(i)}[\Id]\bx_{i,-i}\right\|_2
\end{align*}
By Gaussian concentration, we have:
\[
    \E\left\|\kappa_4^{(i)}[\Id]\bW_{i,-i}\right\|_2^p \le c_p N^{3p/8},
\]
and by using $\|M\|_{\mathrm{op}} \le \|M\|_F$ for any matrix $M$ combined with $\norm*{\bx_{i,-i}} \le \sqrt{N}$, we have:
\[
    \frac{1}{N^p}\E\left\| \kappa_4^{(i)}[\Id] \bx_{i,-i} \right\|_2^p \le O\parens*{N^{3p/8}}\mper
\]
Putting the above together gives
\[
    \E\left\|\kappa_4^{(i)}[\Id]\bM_{i,-i}\right\|_2^p \le O\parens*{ N^{3p/8} }\mper
\]

\smallskip
\noindent \emph{The third, fourth, and fifth terms.}
We have:
\begin{align*}
    \operatorname{Tr}
        \bigl(\kappa_{2,+}^{(i)}-\kappa_{2,-}^{(i)}\bigr)
    &=\beta\int_{-1}^1
        \kappa_{3,s}^{(i)}
            [\Id\otimes\bM_{i,-i}]\,ds\mcom\\
    \kappa_{2,+}^{(i)}-\kappa_{2,-}^{(i)}
    &=\beta\int_{-1}^1
        \kappa_{3,s}^{(i)}[\bM_{i,-i}]\,ds\mcom
\end{align*}
where both quantities on the right-hand side have the $p$-th moment of their Frobenius norm bounded by $O\parens*{N^{p/4}}$ by \Cref{lem:cumulant-contraction-estimates}.
Using $|m_i|,v_i\le1$, and $\E\|d_i\|_2^p\le c_p$, we obtain:
\begin{align*}
    \E\left\| v_i \cdot
        \operatorname{Tr}
            \bigl(\kappa_{2,+}^{(i)}-\kappa_{2,-}^{(i)}\bigr)d_i\right\|_2^p
       &\le O(N^{p/4}),\\
    \E\left\|\frac{2m_iv_i}{N}(1+\|d_i\|_2^2)d_i\right\|_2^p
       &\le O(1),\\
    \E \left\|v_i
        \bigl(\kappa_{2,+}^{(i)}-\kappa_{2,-}^{(i)}\bigr)d_i\right\|_2^p
       &=O(N^{p/4}).
\end{align*}

\smallskip
\emph{Final bound.}
Combining the five estimates along with Markov's inequality on the $p$-th moment for a constant $p$ depending on $a$ proves
\[
    \frac2N
       \left\|\Cov^{(i)}m^{(i)}-(\Cov m)_{-i}\right\|_2
       =O\parens*{\frac{1}{\sqrt{N}}}
\]
with probability $1-O(N^{-a})$.
Since $\beta^2 m_i \le 1$, the desired bound follows.

\smallskip

\noindent \textbf{Bound on $\left\|{ \frac{\beta^3}{3}\kappa_4^{(i)}\bracks*{\bM_{i,-i}^{\otimes 3}} }\right\|_2$.}
We can write:
\begin{align*}
    \kappa_4^{(i)}\bracks*{ \bM_{i,-i}^{\otimes 3} } =& \frac{3}{N} \kappa_4^{(i)}[\Id]\bM_{i,-i} + \kappa_4^{(i)}\bracks*{ \bW_{i,-i}^{\otimes 3} } - \frac{3}{N} \kappa_4^{(i)}[\Id]\bW_{i,-i} \\
    &+ \frac{3\beta\bx_i}{N}\kappa_4^{(i)}\bracks*{(\bW_{i,-i}^{\otimes 2} - N^{-1}\Id)\otimes\bx_{-i}} + \frac{3\beta^2}{N^2} \kappa_4^{(i)}\bracks*{ \bW_{i,-i} \otimes \bx_{-i}^{\otimes 2} } + \frac{\beta^3}{N^3} \kappa_4^{(i)}\bracks*{ \bx_{-i}^{\otimes 3} }
\end{align*}
Using \Cref{lem:improved-cumulant}, our bound on moments of $\kappa_4^{(i)}[\Id]\bM_{i,-i}$, and moment bounds for Gaussians, one may see that:
\[
    \E \left\| \kappa_4^{(i)}\bracks*{ \bM_{i,-i}^{\otimes 3} } \right\|_2^p \le O\parens*{ N^{-5p/8} }\mper
\]
An application of Markov's inequality for an appropriate choice of $p$ implies the desired bound.

\smallskip

\noindent \textbf{Bound on $\left\| \beta^2 m_i \kappa_3^{(i)}\bracks*{ \bM_{i,-i}^{\otimes 2} - \frac{1}{N}\Id } \right\|_2$.}
First, observe that we may write:
\[
    \bM_{i,-i}^{\otimes 2} - \frac{1}{N}\Id = \parens*{\bW_{i,-i}^{\otimes 2} - \frac{1}{N}\Id} + \frac{\beta\bx_i}{N}(\bx_{-i}\otimes \bW_{i,-i} + \bW_{i,-i}\otimes \bx_{-i}) + \frac{\beta^2}{N^2} \bx_{-i}^{\otimes 2}\mper
\]
To analyze contracting $\kappa_3^{(i)}$ via each of the above terms, we find it convenient to define the following matrix $V$ constructed by choosing
\[
    V_{a,b}\coloneqq \angles*{ \kappa_3^{(i)}[e_a], \kappa_3^{(i)}[e_b] }\mper
\]
By symmetry of the cumulant tensor, one can also see that $V = \sum_a \left(\kappa_3^{(i)}[e_a]\right)^2$.
A key identity related to $V$ that can be checked by direct computation is that for any vector $u$, $\left\|\kappa_3^{(i)}[u]\right\|_F = \sqrt{u^{\top}Vu}$.

\noindent \emph{First term.}
A standard computation proves that $\left\|\kappa_3^{(i)}\bracks*{ \bW_{i,-i}^{\otimes 2} - \frac{1}{N}\Id }\right\|_F^2$ is equal to $\frac{2}{N^2}\Tr V$.
To get a handle on this random variable, we calculate moments of its fluctuation, i.e. moments of
\[
    \left\|\kappa_3^{(i)}\bracks*{ \bW_{i,-i}^{\otimes 2} - \frac{1}{N}\Id }\right\|_F^2 - \frac{2}{N^2}\Tr V\mper
\]
Since the above is a degree-$4$ polynomial of gaussian random variables, we may use hypercontractivity to bound the $p$-th moment in terms of its variance.
We may compute its variance using Wick's formula and obtain:
\begin{align*}
    \E&\parens*{ \left\|\kappa_3^{(i)}\bracks*{ \bW_{i,-i}^{\otimes 2} - \frac{1}{N}\Id }\right\|_F^2 - \frac{2}{N^2}\Tr V }^2\\
    &= \frac{1}{N^4}\parens*{ 8 \sum_{a,b} \angles*{\kappa_3^{(i)}[e_a], \kappa_3^{(i)}[e_b] }^2 + 16 \sum_{a,b} \Tr\parens*{ \kappa_3^{(i)}[e_a] \kappa_3^{(i)}[e_b] \kappa_3^{(i)}[e_a] \kappa_3^{(i)}[e_b] } + 32 \Tr V^2 } \\
    &\le \frac{56}{N^4} \Tr V^2 \le O\parens*{ \frac{\Tr V\cdot \|V\|_{\mathrm{op}}}{N^4} }\mper
\end{align*}
where in the first inequality we used that for any pair of symmetric matrices $X,Y$ we have $\Tr XYXY \le \Tr X^2 Y^2$.
By a similar argument to the proof of \eqref{eq:third-cumulant-bound} and \Cref{lem:improved-cumulant} respectively, we see that:
\[
    \E\|V\|_{\mathrm{op}}^p \le O(N^{3p/4}), \qquad \E (\Tr V)^p \le O(N^p)\mper
\]
By the above, we see that the $p$-th moment of the fluctuation for fixed constant $p$ is $O(N^{-9p/8})$.

Observe that $\Tr V = \left\| \kappa_3^{(i)} \right\|_F^2$, which is $O(N)$ with probability probability $1-O\parens*{N^{-a}}$ for any constant $a$, and so we have:
\[
    \left\|\kappa_3^{(i)}\bracks*{ \bW_{i,-i}^{\otimes 2} - \frac{1}{N}\Id }\right\|_F \le O\parens*{\frac{1}{\sqrt{N}}}\mper
\]

\smallskip

\noindent \emph{Second and third terms.}
We now treat the second and third terms.
\begin{align*}
    \E_{\bW}\left\|\frac{\beta\bx_i}{N}\kappa_3^{(i)}\bracks*{ \bx_{-i}\otimes \bW_{i,-i} }\right\|_F^{p} \le \frac{c_p}{N^{3p/2}} \| \kappa_3^{(i)}[\bx_{-i}] \|_F^p = \frac{c_p}{N^{3p/2}} \parens*{ \bx_{-i}^{\top} V \bx_{-i} }^{p/2} \le \frac{c_p}{N^p} \|V\|_{\mathrm{op}}^{p/2}\mper \\
    \left\| \frac{\beta^2}{N^2} \kappa_3^{(i)}\left[\bx_{-i}^{\otimes 2}\right] \right\|_F = \frac{\beta^2}{N^2} \left\|\kappa_3^{(i)}[\bx_{-i}]\bx_{-i}\right\|_F \le \frac{\beta^2}{N^{3/2}} \left\|\kappa_3^{(i)}[\bx_{-i}]\right\|_F = \frac{\beta^2}{N^{3/2}} \sqrt{ \bx_{-i}^{\top} V \bx_{-i} } \le \frac{\beta^2}{N} \|V\|_{\mathrm{op}}^{1/2}\mper   
\end{align*}
The desired bound follows from aforementioned bounds on moments of $\|V\|_{\mathrm{op}}$.

\smallskip

\noindent \textbf{Bound on $\lambda_i$.}
We may express $\lambda_i$ exactly as:
\[
    \lambda_i = \frac{\beta^2}{N}\parens*{ \Tr\Cov - \Tr\Cov^{(i)} } - \beta^2 \parens*{ \bM_{i,-i}^{\top} \Cov^{(i)} \bM_{i,-i} - \frac{1}{N}\Tr\Cov^{(i)} } - \beta \bM_{i,-i}^{\top} (d_i - \beta \Cov^{(i)} \bM_{i,-i})\mper
\]
We now bound each of the above three terms.

Observe that the first term is:
\begin{align*}
    |\Tr\Cov - \Tr \Cov^{(i)}| &= \left| 1 - m_i^2 + \left\|m^{(i)}\right\|^2 - \left\| m_{-i} \right\|^2 \right| \\
    &\le 1 + \left| \angles*{ m^{(i)} - m_{-i}, m^{(i)} + m_{-i} } \right| \\
    &\le 1 + 2\sqrt{N} \left\| m^{(i)} - m_{-i} \right\|\mper
\end{align*}
We have that:
\[
    m_{-i} - m^{(i)} = \frac{1+m_i}{2}\parens*{ m^{(i,+)} - m^{(i)} } + \frac{1-m_i}{2}\parens*{ m^{(i,-)} - m^{(i)} }\mcom
\]
By a Taylor expansion, we see:
\[
    m^{(i,\pm)} - m^{(i)} = \pm\beta\Cov^{(i)} \bM_{i,-i} + \beta^2 \int_0^1 (1-s) \kappa_{3,\pm s}^{(i)}\bracks*{ \bM_{i,-i}^{\otimes 2} }\, ds\mper
\]
By \Cref{lem:cumulant-contraction-estimates}, the $p$-th moment of the above for fixed $p$ is $O(1)$.
Thus,
\[
    \E \parens*{ \frac{\beta^2}{N} \parens*{\Tr\Cov - \Tr\Cov^{(i)}} }^p \le O\parens*{N^{-p/2}}\mper
\]

For the second term, observe that:
\begin{align*}
    &\bM_{i,-i}^{\top}\Cov^{(i)}\bM_{i,-i}
        -\frac1N\Tr\Cov^{(i)}\\
    &\qquad=
    \left(
        \bW_{i,-i}^{\top}\Cov^{(i)}\bW_{i,-i}
        -\frac1N\Tr\Cov^{(i)}
    \right)
    +\frac{2\beta\bx_i}{N}
        \bW_{i,-i}^{\top}\Cov^{(i)}\bx_{-i}
    +\frac{\beta^2}{N^2}
        \bx_{-i}^{\top}\Cov^{(i)}\bx_{-i}.
\end{align*}
By bounds on gaussian moments, we have:
\begin{align*}
    \left(
        \E_{\bW}\left|
            \bW_{i,-i}^{\top}\Cov^{(i)}\bW_{i,-i}
            -\frac1N\Tr\Cov^{(i)}
        \right|^p
    \right)^{1/p}
    &\le \frac{C_p}{N}\|\Cov^{(i)}\|_F,\\
    \left(
        \E_{\bW}\left|
            \frac{2\beta\bx_i}{N}
            \bW_{i,-i}^{\top}\Cov^{(i)}\bx_{-i}
        \right|^p
    \right)^{1/p}
    &\le
    \frac{C_p}{N^{3/2}}\|\Cov^{(i)}\bx_{-i}\|_2
    \le\frac{C_p}{N}\|\Cov^{(i)}\|_F.
\end{align*}
By Cauchy--Schwarz, we have:
\[
    \frac{\beta^2}{N^2}
    \left|\bx_{-i}^{\top}\Cov^{(i)}\bx_{-i}\right|
    \le \frac{\beta^2}{N}\|\Cov^{(i)}\|_F.
\]
Thus, by \Cref{lem:cumulant-bounds} we may conclude that:
\[
    \E\left|
        \beta^2\left(
            \bM_{i,-i}^{\top}\Cov^{(i)}\bM_{i,-i}
            -\frac1N\Tr\Cov^{(i)}
        \right)
    \right|^p
    \le C_pN^{-p/2}.
\]
For the third term, first observe that by performing a Taylor expansion we get:
\[
    \beta \bM_{i,-i}^{\top} \parens*{d_i - \beta\Cov^{(i)}\bM_{i,-i}} = \frac{\beta^4}{4}\int_0^1 (1-s)^2 \parens*{ \kappa_{4,s}^{(i)} + \kappa_{4,-s}^{(i)} }\bracks*{ \bM_{i,-i}^{\otimes 4} }ds\mper
\]
By \Cref{lem:cumulant-contraction-estimates}, we get:
\begin{equation}\label{eq:lambda_i-Lp-bound}
    \E\left|\beta \bM_{i,-i}^{\top} \parens*{d_i - \beta\Cov^{(i)}\bM_{i,-i}}\right|^p \le O\parens*{N^{-p}}\mper
\end{equation}
Combining the estimates for the three terms gives:
\[
    \E|\lambda_i|^p \le O\parens*{N^{-p/2}}\mcom
\]
and Markov's inequality along with a union bound for appropriate choice of $p$ proves the desired statement.

\subsection{Proof of \Cref{lem:cumulant-contraction-estimates}}\label{subsec:proof-of-cumulant-contraction}
We will prove the statement for $\kappa_{r,s}^{(i)}$.
The analogous statement for $\kappa_{r,s}^{(B)}$ follows an identical proof.

Fix $s\in [-1,1]$. Recall that the goal is to prove:
\[
    \E \left\| \kappa_{r,s}^{(i)}\bracks*{ \Id^{\otimes k} \otimes \bM_{i,-i}^{\otimes \ell} } \right\|_F^{p} \le C_p N^{p(r/4 - \ell/2)}\mper
\]
First, observe that we may write:
\[
    \left\| \kappa_{r,s}^{(i)}\bracks*{ \Id^{\otimes k} \otimes \bM_{i,-i}^{\otimes \ell} } \right\|_F \le C_{\ell} \sum_{\ell_1+\ell_2=\ell} \left\| \kappa_{r,s}^{(i)}\bracks*{ \Id^{\otimes k} \otimes \bW_{i,-i}^{\otimes \ell_1} \otimes \frac{\bx_{-i}^{\otimes \ell_2}}{N^{\ell_2}} } \right\|_F
\]
for some $C_{\ell}$ that only depends on $\ell$.
To obtain a handle on the above (up to a loss of a constant depending on $p$ and $\ell$), it suffices to study a single summand.

An ingredient in obtaining a bound on a single summand is the following bound in the special case when $\ell_1 = 0$, whose proof we defer until the end of this proof.
\begin{restatable}{lemma}{lemplantcumulant}
    \label{lem:plant-cumulant}
    For every fixed $r\ge 3$, $p\ge 2$, and nonnegative integers $k$ and $\ell$ such that $2k + \ell\le r$, we have:
    \[
        \left( \E\left\| \kappa_{r,s}^{(i)}\bracks*{ \Id^{\otimes k} \otimes \bx_{-i}^{\otimes \ell} } \right\|_{\mathrm F}^p \right)^{1/p} \le C N^{(r+\ell)/4}\mper
    \]
\end{restatable}

In particular, we perform a proof by induction, where the above is the base case.
Henceforth, for the induction hypothesis, assume that there is some choice of $\ell_1\ge 0$ for which the desired statement holds.

In the rest of this argument $\ell_2$ is fixed, and  we denote the cumulant tensor of interest by
\[
    \kappa_r(k,\ell_1) \coloneqq \kappa_{r,s}^{(i)}\bracks*{ \Id^{\otimes k} \otimes \bW_{i,-i}^{\otimes \ell_1} \otimes \frac{\bx_{-i}^{\otimes \ell_2}}{N^{\ell_2}} }\mper
\]
We split the term of interest into a sum of its mean and a fluctuation term, and control each term separately.
In particular, we have:
\[
    \kappa_r(k,\ell_1) = \E_{\bW_{i,-i}} \kappa_r(k,\ell_1) + \parens*{ \kappa_r(k,\ell_1) - \E_{\bW_{i,-i}} \kappa_r(k,\ell_1) }\mper
\]

\noindent \textbf{ Controlling $\E \kappa_r(k,\ell_1)$. }
Observe that by gaussian integration by parts, we have:
\[
    \E \kappa_r(k,\ell_1) = \frac{\ell_1-1}{N}\E \kappa_r(k+1,\ell_1-2) + \frac{s\beta}{N}\E\kappa_{r+1}(k+1,\ell_1-1)\mcom
\]
following the convention that $\kappa_r(k,-1) = 0$ for any $r,k$.
By treating \Cref{lem:plant-cumulant} as a base case for an induction argument, we obtain:
\[
    \E_{\bW_{-i,-i},\bx} \left\| \E_{\bW_{i,-i}} \kappa_{r,s}^{(i)}\bracks*{ \Id^{\otimes k} \otimes \bW_{i,-i}^{\otimes \ell_1} \otimes \frac{\bx_{-i}^{\otimes \ell_2}}{N^{\ell_2}} } \right\|_F^p \le C_p N^{p(r/4-3\ell_2/4-\ell_1/2)}\mper \numberthis \label{eq:mean-term}
\]
It remains to control the fluctuation term.

\smallskip

\noindent \textbf{Controlling $\kappa_r(k,\ell_1) - \E \kappa_r(k,\ell_1)$.}
To control the $p$-th moment of the fluctuation, we treat it as a function of $\bW_{i,-i}$, and apply the $L^p$ Gaussian \Poincare inequality, which says that for any smooth function $F:\R^m\to\R^n$, we have:
\[
    \E_{\bg}\left\|F(\bg)-\E F(\bg)\right\|_F^p \le C_p \E_{\bg} \left\|\grad F(\bg)\right\|_F^p\mper
\]
As a result, we have:
\[
    \E_{\bW_{i,-i}} \left\| \kappa_r(k,\ell_1) - \E \kappa_r(k,\ell_1) \right\|_F^p \le \frac{C_p}{\sqrt{N}^p} \E_{\bW_{i,-i}} \left\| \grad \kappa_r(k,\ell_1) \right\|_F^p\mper \numberthis \label{eq:gaussian-poincare-ellp}
\]
Now, we may compute the gradient as:
\[
    \grad \kappa_r(k,\ell_1) = \ell_1 \kappa_r(k,\ell_1-1) + s\beta \kappa_{r+1}(k,\ell_1)    \numberthis \label{eq:gradient-recursion}
\]
By \Cref{eq:mean-term,eq:gaussian-poincare-ellp,eq:gradient-recursion}, along with the induction hypothesis, we have:
\[
    \E \left\| \kappa_r(k,\ell_1) - \E \kappa_r(k,\ell_1) \right\|_F^p \le C_p N^{p(r/4-3\ell_2/4-\ell_1/2)} + \frac{C_p}{\sqrt{N}^p} \E \left\|\kappa_{r+1}(k,\ell_1)\right\|_F^p
\]
Iterating the above inequality $2\ell_1$ times, we obtain:
\[
    \E \left\| \kappa_r(k,\ell_1) - \E \kappa_r(k,\ell_1) \right\|_F^p \le C_p N^{p(r/4 - 3\ell_2/4 - \ell_1/2)} + \frac{C_p}{N^{p\ell_1}} \E\left\| \kappa_{r+2\ell_1}(k,\ell_1) \right\|_F^p
\]
By Cauchy-Schwarz, we may bound the second term as:
\[
    \E\left\| \kappa_{r+2\ell_1}(k,\ell_1) \right\|_F^p \le \E\left\| \kappa_{r+2\ell_1}\left[\Id^{\otimes k}\otimes \frac{\bx_{-i}^{\otimes \ell_2}}{N^{\ell_2}} \right] \right\|_F^p \cdot \left\| \bW_{i,-i} \right\|_F^{\ell_1 p}\mper
\]
By another application of Cauchy-Schwarz, \Cref{lem:plant-cumulant}, and Gaussian concentration, we may bound the above by $C_p N^{p(r/4-\ell_1/2-3\ell_2/4)}$, which completes the proof.

We are now ready to prove \Cref{lem:plant-cumulant}.

\lemplantcumulant*

\begin{proof}
For simplicity, assume that $k+\ell > 0$; the same ideas extend to the setting where $k = \ell = 0$. We write
$Y\coloneqq
\kappa_{r,s}^{(i)}
[\Id^{\otimes k}\otimes\bx_{-i}^{\otimes \ell}]$. For ease of notation, set $d = r - 2k-\ell$.
We drop the cavity and shift notation, denoting by $F$ the corresponding free energy and by $\bx$ the cavity vector $\bx_{-i}$. Denote $n = N-1$. We now use the definition of $\kappa_r$ as an order-$r$ gradient of $F$ in order to simplify the expression for $Y$. We have, for example, that $(\kappa_3[\Id])_{a} = \sum_{b} \partial_{h_a} \partial_{h_b}^2 F = \partial_{h_{a}} \Lapl F$, so $\kappa_3[\Id] = \grad \Lapl F$. Similarly, contracting by $\bx$ can be expressed in terms of the gradient---for example, $(\kappa_3[\bx])_{ab} = \sum_{c} \partial_{h_{a}} \partial_{h_{b}} \cdot \bx_c \partial_{h_{c}} F$, so $\kappa_3[\bx] = \grad^2 (\bx \cdot \grad) F$. More generally, we have
\[ Y = \grad^{r-2k-\ell} \Lapl^{k} (\bx \cdot \grad)^{\ell} F \mper \]
We will start by establishing that the $k$ Laplacians and $\ell$ inner gradients will eventually result in a homogenous degree-$k+\ell$ polynomial $P$ in centered overlaps, with $Y = N^{k+\ell} \grad^{d} \langle P\rangle$.

Suppose that $k+\ell > 0$. Let us start by understanding how a Laplacian/gradient affects the free energy. For starters, we have
\begin{align*}
    \Lapl F &= \Tr \Cov = n (1 - \langle R_{12} \rangle) = \text{constant} - n \langle \ol{R}_{12} \rangle \mcom \\
    (\bx \cdot \grad) F &= \langle \bx , m \rangle = n \langle R_{1\star} \rangle = \text{constant} + n \langle R_{1\star} \rangle \mper
\end{align*}
Because $r \ge 3$, the constant terms will ultimately vanish.
Suppose that after $m$ contraction operations, we are at some polynomial $n^{m} \langle P_{m} \rangle$. A routine calculation shows that taking the Laplacian/gradient results in the Gibbs average of a polynomial of degree larger by $1$. We execute this calculation for the gradient; the calculation for the Laplacian is similar. Suppose that $P_m$ is some degree-$m$ polynomial in centered overlaps among $v$ replicas (and the plant).
\begin{align*}
    (\bx \cdot \grad) \langle P_{m}(\sigma^1, \ldots, \sigma^v) \rangle &= \left\langle \bx , \left\langle P_{m} \cdot \sum_{a=1}^{v} \left( \sigma^j - \sigma^{v+1} \right) \right\rangle \right\rangle \\
        &= n \cdot \left\langle P_m \cdot \sum_{a=1}^{v} \left( R_{j\star} - R_{(v+1)\star} \right) \right\rangle \\
        &= n \cdot \left\langle P_m \cdot \sum_{a=1}^{v} \left( \ol{R}_{j\star} - \ol{R}_{(v+1)\star} \right) \right\rangle \mper
\end{align*}
This is the Gibbs average of a homogenous degree-$(m+1)$ polynomial in centered overlaps! It follows that when $k+\ell > 0$,
\[ Y = n^{k+\ell} \grad^{r - 2k - \ell} \langle P \rangle \]
for some polynomial $P$ in centered overlaps. We may now use the same rules to expand out the gradient:
\begin{align*}
    \left(\grad^{d} \langle P(\sigma^1,\ldots,\sigma^v) \rangle\right) = \left\langle P(\sigma^1,\ldots,\sigma^v) \cdot \bigotimes_{a=v}^{v+d-1} \left( \sum_{b=1}^{a} \left( \sigma_b - \sigma_{a+1} \right) \right) \right\rangle \mcom
\end{align*}
and thus, using the replica trick to write $\langle G(\sigma)\rangle^2 = \langle G(\sigma) G(\wt{\sigma}) \rangle$ for a fresh replica $\wt{\sigma}$, we have
\begin{align*}
    \left\| Y \right\|_{\mathrm F}^2 &= n^{2(k+\ell)} \sum_{i_1,\ldots,i_d} \left\langle P(\sigma^1,\ldots,\sigma^v) \cdot \prod_{a=v}^{v+d-1} \left( \sum_{b=1}^{a} \left( \sigma_b - \sigma_{a+1} \right) \right)_{i_{a-v+1}} \right\rangle^2 \\
        &= n^{2(k+\ell)} \sum_{i_1,\ldots,i_d} \left\langle P(\sigma^1,\ldots,\sigma^v) \cdot P(\wt{\sigma}^1,\ldots,\wt{\sigma}^v) \cdot \prod_{a=v}^{v+d-1} \left( \sum_{b=1}^{a} \left( \sigma_b - \sigma_{a+1} \right) \right)_{i_{a-v+1}} \left( \sum_{b=1}^{a} \left( \wt{\sigma}_b - \wt{\sigma}_{a+1} \right) \right)_{i_{a-v+1}} \right\rangle \\
        &= n^{2(k+\ell)} \left\langle P(\sigma^1,\ldots,\sigma^v) \cdot P(\wt{\sigma}^1,\ldots,\wt{\sigma}^v) \cdot \prod_{a=v}^{v+d-1} \left\langle \sum_{b=1}^{a} \left( \sigma_b - \sigma_{a+1} \right) , \sum_{b=1}^{a} \left( \wt{\sigma}_b - \wt{\sigma}_{a+1} \right) \right\rangle \right\rangle \mper
\end{align*}
Note now that this final product term may be expressed as a degree-$d$ product of \emph{centered} replicas---this again uses the fact that the sum of the coefficients is $0$. Summarizing, we have
\[ \|Y\|_{F}^2 = n^{2(k+\ell)+d} \cdot \langle P \rangle = n^{r+\ell} \langle P \rangle \mper \]
for some homogenous degree-$2(k+\ell)+d$ polynomial $P$ in centered overlaps. Noting $2k+\ell+d = r$ and using replicas again, we have
\[ \|Y\|_{\mathrm F}^{2p} = n^{p(r+\ell)} \cdot \langle Q \rangle \]
for some homogenous degree-$p(r+\ell)$ polynomial $Q$. To complete the proof, we simply use H\"{o}lder's inequality on each term of the polynomial to reduce bounding it to bounding large moments of centered replicas---this is the content of \Cref{lem:shifted-model-overlap-conc}, which yields that $\E\langle Q\rangle \le C n^{-p(r+\ell)/2}$, and thus that $\E\|Y\|_{\mathrm F}^{2p} \le C N^{p(r+\ell)/2}$.
\end{proof}

\subsection{Proof of \Cref{lem:improved-cumulant}}\label{subsec:proof-of-improved-cumulant}

Throughout this subsection, $B \subseteq [N]$ is a set with $|B| \le 4$; no intuition is lost by considering $B = \emptyset$.
Define the TAP map $T^{(B)}:(-1,1)^{[N] \setminus B}\to \R^{[N] \setminus B}$ by
\[
T^{(B)}_i(m)=\arctanh(m_i)
-\beta\bM_{i,-(B \cup \{i\})}^{\mathsf T}m_{-(B \cup \{i\})}
+\frac{\beta^2}{N}(N-|B|-\|m\|_2^2)m_i.
\]
Let $F^{(B)}(u)$ be the free energy on spins $[N] \setminus B$ with interaction $\beta \bM_{-B,-B}$ and external field $u \in \R^{[N] \setminus B}$.
For $q\ge 1$, let $\kappa^{(B)}_q(u) = \grad^q F^{(B)}(u)$; in particular $\kappa^{(B)}_1(u) = m^{(B)}(u)$ and $\kappa^{(B)}_2(u) = \Cov^{(B)}(u)$ are the mean and covariance of this Gibbs measure with external field $u$.

Recall that the planted model in \Cref{th:cov-bounds} has external field $h = t\bx + \sqrt{t}\bg$.
Whenever we omit the argument $u$ by writing e.g. $\Cov^{(B)}$, we evaluate at $u = h_{-B}$.

Throughout we use $\grad$ to denote derivatives with respect to $u$, while $D$ denotes derivatives with respect to the input $m$ of $T^{(B)}$, evaluated at $m = m^{(B)}(h_{-B})$.
In particular, for $K^{(B)}$ defined in \eqref{eq:inv-cov-pred}, we have $DT^{(B)}(m) = K^{(B)}$.

The TAP heuristic predicts that $T^{(B)}(m^{(B)}(u))\approx u$, which motivates the definition of the residual
\begin{equation}\label{eq:tap-residual}
    r^{(B)}(u)=T^{(B)}(m^{(B)}(u)) - u\mper
\end{equation}
The starting point of the proof of \Cref{lem:improved-cumulant} is the following set of identities, which are proved directly by differentiating \Cref{eq:tap-residual} with respect to $u$.
\begin{fact}\label{fac:improved-cumulant-start}
We have
    \begin{align*}
    K^{(B)}\Cov^{(B)} - \Id
    &=\grad r^{(B)},\\
    K^{(B)}\kappa^{(B)}_3
    &=\grad ^2r^{(B)}-D^2T^{(B)}(m)\{\Cov^{(B)},\Cov^{(B)}\},\\
    K^{(B)}\kappa^{(B)}_4
    &=\grad^3r^{(B)}
    -\sum_{\Sym} D^2T^{(B)}(m)\{\kappa^{(B)}_3,\Cov^{(B)}\}
    -D^3T^{(B)}(m)\{\Cov^{(B)},\Cov^{(B)},\Cov^{(B)}\}.
    \end{align*}    
\end{fact}
\begin{remark}
To clarify the notation, for $q=2,3,4$, $\kappa^{(B)}_q$ is a tensor of order $q$, and $K^{(B)}$ acts by matrix multiplication on the first coordinate of $\kappa^{(B)}_q$: 
\[
(K^{(B)}\kappa^{(B)}_q)_{i j_1\cdots j_{q-1}}
= \sum_{a\notin B} K^{(B)}_{ia}(\kappa^{(B)}_q)_{a j_1\cdots j_{q-1}}.
\]
For $q = 1,2,3$, $\grad^q r^{(B)}$ is a tensor of order $q+1$, where the first coordinate indexes the entries of $r^{(B)}$ while the remaining $q$ coordinates index the $q$ derivatives.
Similarly, $D^q T^{(B)}(m)$ is a tensor of order $q+1$ where the first coordinate indexes the entries of $T^{(B)}$.
In the notation $D^2 T^{(B)}(m)\{\Cov^{(B)},\Cov^{(B)}\}$, the matrices $\Cov^{(B)}$ act by matrix multiplication on all but the first coordinate of $D^2 T^{(B)}(m)$:
\[
\left(D^2T^{(B)}(m)\{\Cov^{(B)},\Cov^{(B)}\}\right)_{ijk}
=
\sum_{a,b\notin B}
\frac{\partial^2T_i^{(B)}}{\partial z_a\partial z_b}(m)
\,\Cov^{(B)}_{aj}\Cov^{(B)}_{bk},
\]
We define $D^3 T^{(B)}(m) \{\Cov^{(B)},\Cov^{(B)},\Cov^{(B)}\}$ similarly.
In $D^2T^{(B)}(m)\{\kappa^{(B)}_3,\Cov^{(B)}\}$, $\kappa^{(B)}_3$ and $\Cov^{(B)}$ contract with all but the first coordinates of $D^2 T^{(B)}(m)$:
\[
    \left(D^2T^{(B)}(m)\{\kappa^{(B)}_3,\Cov^{(B)}\}\right)_{ijk\ell}
    = \sum_{a,b\notin B}
    \frac{\partial^2T_i^{(B)}}{\partial z_a\partial z_b}(m)
    (\kappa^{(B)}_3)_{ajk}\Cov^{(B)}_{b\ell}\mper
\]
Finally, $\sum_{\Sym}$ sums cyclically over the last three coordinates of the resulting $4$-tensor:
\[
    \left(\sum_{\Sym}D^2T^{(B)}(m)\{\kappa^{(B)}_3,\Cov^{(B)}\}\right)_{ijk\ell}
    = \sum_{abc \in \{jk\ell, k\ell j, \ell jk\}}
    \left(D^2T^{(B)}(m)\{\kappa^{(B)}_3,\Cov^{(B)}\}\right)_{iabc}
\]
\end{remark}
The main step in the proof of \Cref{lem:improved-cumulant} is to obtain suitably sharp $L^p$ bounds on the Frobenius norms of all the right-hand sides in \Cref{fac:improved-cumulant-start}.
In particular, for $|B|\le 3$, this can be translated into $L^p$ bounds on $\|\kappa^{(B)}_3\|_F$, $\|\kappa^{(B)}_4\|_F$, and $\|\Cov^{(B)}\|_{\mrop}$: on the event $K^{(B)} \succeq \gamma$, which holds with overwhelming probability by \Cref{lem:organized-gap}, these are bounded by $\gamma^{-1} \|K^{(B)}\kappa^{(B)}_3\|_F$, $\gamma^{-1} \|K^{(B)}\kappa^{(B)}_4\|_F$, and $\gamma^{-1} \|K^{(B)}\Cov^{(B)}\|_{\mrop}$, while the $L^p$ contribution from the event $K^{(B)} \not\succeq \gamma$ can be controlled by crude polynomial bounds on $\|\kappa^{(B)}_3\|_F$, $\|\kappa^{(B)}_4\|_F$, and $\|\Cov^{(B)}\|_{\mrop}$.

For $|B|\le 3$ and $i\in[N]\setminus B$, recall from \Cref{sec:cov-bound} that $v_i^{(B)} = 1 - (m_i^{(B)})^2$, $q_N^{(B)} = \|m^{(B)}\|^2/N$, and
\[
    d_i^{(B)} = \frac{\Cov^{(B)}_{i,-(B\cup\{i\})}}{v_i^{(B)}} = \frac{m^{(B,i,+)} - m^{(B,i,-)}}{2}, \qquad
    \lambda_i^{(B)} = \beta^2\left(1-\frac{|B|}{N}-q_N^{(B)}\right) - \beta \bM_{i,-(B\cup\{i\})}^\top d_i^{(B)}.
\]
Introduce the notation
\begin{equation}\label{e:phi-i}
    \phi_i^{(B)}(s) = F^{(B\cup\{i\})}(u^{(B\cup\{i\})} + s\beta \bM_{i,-(B\cup\{i\})})\mper
\end{equation}
Also for $s\in \{\pm 1\}$, define
\[
    m^{(B,i,s)} = \grad F^{(B\cup\{i\})}(u^{(B\cup\{i\})} + s\beta \bM_{i,-(B\cup\{i\})})\mcom
\]
which is the mean of the Gibbs measure on spins $[N] \setminus (B \cup \{i\})$, where spins $\{\sigma_j : j\in B\}$ are removed and $\sigma_i$ is set to $s$.
The following lemma provides a description of $r^{(B)}$, which will be useful for controlling its derivatives appearing in \Cref{fac:improved-cumulant-start}.
\begin{lemma}\label{lem:r-decomp}
    For $|B|\le 3$ and $i\in[N]\setminus B$, we have $r_i^{(B)} = \zeta_i^{(B)} + m_i^{(B)} \lambda_i^{(B)}$, where
    \begin{equation}\label{eq:zeta}
        \zeta_i^{(B)}
        = \frac{\phi_i^{(B)}(1) - \phi_i^{(B)}(-1) - (\phi_i^{(B)})'(1) - (\phi_i^{(B)})'(-1)}{2}
        = -\frac14 \int_{-1}^1 (1-s^2) (\phi_i^{(B)})^{(3)}(s) \dif s
    \end{equation}
\end{lemma}
\begin{proof}
    We provide the proof for $B = \emptyset$ and omit the superscript $(\emptyset)$; the case of general $B$ is analogous.
    For $s\in \{\pm 1\}$, the contribution $Z_i(s)$ to the partition function from $\sigma_i = s$ is $\exp(h_i s + \phi_i(s))$.
    Thus,
    \begin{equation}\label{eq:one-step-conditioning-new}
        \arctanh(m_i) - h_i
        = \frac12 \log \frac{1+m_i}{1-m_i} - h_i
        = \frac12 \log \frac{Z_i(+1)}{Z_i(-1)} - h_i
        = \frac{\phi_i(+1) - \phi_i(-1)}{2}.
    \end{equation}
    Differentiating \Cref{e:phi-i} yields, for $s\in \{\pm 1\}$,
    \[
        \phi'_i(s) = \beta \bM_{i,-i}^\top m^{(i,s)}\mper
    \]
    Combining with $m_{-i} = \frac{1+m_i}{2} m^{(i,+)} + \frac{1-m_i}{2} m^{(i,-)}$ yields
    \[
        \beta \bM_{i,-i}^\top m_{-i} = \frac{1+m_i}{2} \phi'_i(1) + \frac{1-m_i}{2} \phi'_i(-1)
        = \frac{\phi'_i(1) + \phi'_i(-1)}{2} + m_i \cdot \beta \bM^\top_{i,-i} d_i\mper
    \]
    Substituting the last display and \Cref{eq:one-step-conditioning-new} into $r_i = T_i(m) - h_i$ yields $r_i = \zeta_i + m_i \lambda_i$.
    The second equality in \Cref{eq:zeta} follows from integration by parts.
\end{proof}
Using \cref{lem:r-decomp}, the right-hand sides in \cref{fac:improved-cumulant-start} can be expanded as a combination of $d_i^{(B)}$ and derivatives in $m_i^{(B)}$, $\zeta_i^{(B)}$, $\lambda_i^{(B)}$.
In the following lemma, we collect the $L^p$ estimates we need for these quantities.
The proof is largely by Taylor expanding these quantities in terms of contractions of the $\kappa_{r,s}^{(i)}$ and applying \cref{lem:cumulant-contraction-estimates}; we defer the proof to the end of this subsection.
\begin{lemma}\label{lem:tap-derivative-estimates-new}
    For every fixed $p\ge 2$, there exists $C_p$ such that for all $|B|\le 3$ and $i\in[N]\setminus B$,
    \[
        \E \|d_i^{(B)} - \beta \Cov^{(B \cup \{i\})} \bM_{i, -(B\cup\{i\})}\|_2^p
        \le C_p N^{-p/2}\mcom
    \]
    and for $j=1,2,3$,
    \begin{align*}
        \E \|(v_i^{(B)})^{-1} \grad^j m_i^{(B)}\|_F^p &\le C_p N^{p(j-1)/4}, \\
        \E \|\grad^j \zeta_i^{(B)}\|_F^p &\le C_p N^{p(j-3)/4}.
    \end{align*}
    Finally, for $j=0,1,2,3$,
    \[
        \E \|\grad^j \lambda_i^{(B)}\|_F^p \le C_p \left(
            N^{p(j-3)/4} 
            + N^{-p} \E \|\kappa^{(B\cup\{i\})}_{j+2}\|_F^p 
        \right)\mper
    \]
    and in particular $\E |\lambda_i^{(B)}|^p \le C_p N^{-p/2}$.
\end{lemma}

\begin{corollary}\label{cor:r-grad23}
    For all $p\ge 2$, there exists $C_p > 0$ such that for all $|B|\le 3$,
    \[
        \E \|\grad^2 r^{(B)}\|_F^p \le C_p N^{p/2}, \qquad
        \E \|\grad^3 r^{(B)}\|_F^p \le C_p N^{3p/4}.
    \]
\end{corollary}
\begin{proof}
    By \cref{lem:cumulant-bounds}, for each $j=1,2,3$,
    \[
        \E \|\kappa^{(B\cup\{i\})}_{j+2}\|_F^p \le C_p N^{p(j+2)/4}\mcom
    \]
    and so the estimate on $\grad^j \lambda_i^{(B)}$ in \cref{lem:tap-derivative-estimates-new} gives
    \[
        \E \|\grad^j \lambda_i^{(B)}\|_F^p \le C_p N^{p(j-2)/4}
    \]
    after adjusting $C_p$.
    Recall from \cref{lem:r-decomp} that $r_i^{(B)} = \zeta_i^{(B)} + m_i^{(B)} \lambda_i^{(B)}$.
    The result follows by differentiating this identity and applying the last display and the remaining estimates from \cref{lem:tap-derivative-estimates-new}.
\end{proof}
The following lemma directly implies \cref{lem:improved-cumulant}\ref{it:improved-cumulant-kappa34}, which corresponds to the case $|B| \le 1$; we allow more general $B$ for use in the rest of the proof of \cref{lem:improved-cumulant}.
It also provides stronger bounds on $\|\kappa_3^{(B)}\|_F$ and $\|\kappa_4^{(B)}\|_F$ than obtained in \cref{lem:cumulant-bounds}, which strengthens the estimate on $\grad^j \lambda_i^{(B)}$ in \cref{lem:tap-derivative-estimates-new}; this will be useful in \cref{cor:r-grad12} below.
\begin{lemma}\label{lem:improved-cumulant-kappa34}
    For all $p\ge 2$, there exists $C_p > 0$ such that for all $|B| \le 3$,
    \[
        \E \|\kappa_3^{(B)}\|_F^p
        \le C_p N^{p/2}\mcom \qquad
        \E \|\kappa_4^{(B)}\|_F^p
        \le C_p N^{3p/4}\mper
    \]
    In particular, \cref{lem:improved-cumulant}\ref{it:improved-cumulant-kappa34} follows as the case $|B| \le 1$.
\end{lemma}
\begin{proof}
    We provide the proof for $B = \emptyset$ and omit the superscript $(\emptyset)$; the case of general $B$ is analogous.
    Most of the proof will be devoted to the estimates
    \begin{align}
        \label{eq:D2-cov-cov}
        \E \|D^2T(m)\{\Cov,\Cov\}\|_F^p &\le C_p N^{p/2}\mcom \\
        \label{eq:D2-k3-cov}
        \E \|D^2T(m)\{\kappa_3,\Cov\}\|_F^p &\le C_p N^{3p/4}\mcom \\
        \label{eq:D3-cov-cov-cov}
        \E \|D^3T(m)\{\Cov,\Cov,\Cov\}\|_F^p &\le C_p N^{p/2}\mper
    \end{align}
    We first explain why they imply the conclusion.
    Together with \cref{cor:r-grad23}, these imply
    \[
        \E \|K\kappa_3\|_F^p
        \le C_p N^{p/2}\mcom \qquad
        \E \|K\kappa_4\|_F^p
        \le C_p N^{3p/4}\mper
    \]
    Let $\calE = \{K\succeq\gamma\Id\}$, where $\gamma>0$ is given by \cref{lem:organized-gap}.
    On $\calE$, for each $q\in \{3,4\}$,
    \[
        \|\kappa_q\|_F
        \le \|K^{-1}\|_{\mrop} \|K\kappa_q\|_F
        \le \gamma^{-1} \|K\kappa_q\|_F\mper
    \]
    while on the complement of $\calE$, we can use the crude bound $\|\kappa_q\|_F \le CN^{q/2}$.
    By the above and \cref{lem:organized-gap},
    \[
        \E \|\kappa_3\|_F^p
        \le \gamma^{-p} \E \|K\kappa_3\|_F^p
        + \Pr[\calE^c] C_p N^{3p/2}
        \le C_p N^{p/2}\mcom
    \]
    and similarly $\E \|\kappa_4\|_F^p \le C_p N^{3p/4}$.
    We now turn to the proofs of \cref{eq:D2-cov-cov,eq:D2-k3-cov,eq:D3-cov-cov-cov}.
    Write
    \begin{equation}\label{eq:AU}
        A(z) = \arctanh(z)\mcom \qquad
        U(z) = - \frac{\beta^2}{N} \|z\|^2 z\mcom
    \end{equation}
    where $\arctanh$ is applied entry-wise.
    Note that $T - A - U$ is linear in $m$.
    It therefore suffices to establish \cref{eq:D2-cov-cov,eq:D2-k3-cov,eq:D3-cov-cov-cov} with $A$ and $U$ in place of $T$.

    We first treat $A$.
    Set
    \[
        y_i=v_i^{-1}\grad m_i\mcom\qquad
        z_i=v_i^{-1}\grad^2m_i\mper
    \]
    Since $\grad m = \Cov$ and $\grad^2m = \kappa_3$, direct differentiation gives
    \begin{align*}
        \left( D^2A(m)\{\Cov,\Cov\} \right)_i
        &= 2m_i y_i^{\otimes2}\mcom \\
        \left( D^2A(m)\{\kappa_3,\Cov\} \right)_i
        &= 2m_i z_i\otimes y_i\mcom \\
        \left( D^3A(m)\{\Cov,\Cov,\Cov\} \right)_i
        &= 2(1+3m_i^2)y_i^{\otimes3}\mper
    \end{align*}
    By \cref{lem:tap-derivative-estimates-new} with $j=1,2$, for every fixed $p\ge 2$,
    \[
        \E\|y_i\|_2^p\le C_p\mcom \qquad
        \E\|z_i\|_F^p\le C_pN^{p/4} \mper
    \]
    Together with the trivial estimate $|m_i| \le 1$ and H\"older's inequality, we obtain for all $i$
    \begin{align*}
        \E \left\|\left( D^2A(m)\{\Cov,\Cov\} \right)_i\right\|_F^p
        &\le C_p, \\
        \E \left\|\left( D^2A(m)\{\kappa_3,\Cov\} \right)_i\right\|_F^p
        &\le C_p N^{p/4}\mcom \\
        \E \left\|\left( D^3A(m)\{\Cov,\Cov,\Cov\} \right)_i\right\|_F^p
        &\le C_p\mper
    \end{align*}
    For any tensor $H \in (\R^N)^{\otimes q}$ with slices $H_i \in (\R^N)^{\otimes (q-1)}$ along the first index, the inequality
    \[
        \E\|H\|_F^p
        =\E\left(\sum_i\|H_i\|_F^2\right)^{p/2}
        \le N^{p/2-1}\sum_i\E\|H_i\|_F^p
    \]
    therefore yields \cref{eq:D2-cov-cov,eq:D2-k3-cov,eq:D3-cov-cov-cov} with $A$ in place of $T$.
    We next treat $U$.
    Direct differentiation gives
    \begin{align}
        \label{eq:D2U}
        \left(D^2U(m)\{\Cov,\Cov\}\right)_{ijk}
        &=-\frac{2\beta^2}{N}\left(
            m_i(\Cov^2)_{jk}+\Cov_{ij}(\Cov m)_k+\Cov_{ik}(\Cov m)_j
        \right),\\
        \nonumber
        \left(D^2U(m)\{\kappa_3,\Cov\}\right)_{ijk\ell}
        &=-\frac{2\beta^2}{N}\left(
            m_i\sum_a(\kappa_3)_{ajk}\Cov_{a\ell}
            +(\kappa_3)_{ijk}(\Cov m)_\ell
            +\Cov_{i\ell}\sum_a m_a(\kappa_3)_{ajk}
        \right),\\
        \nonumber
        \left(D^3U(m)\{\Cov,\Cov,\Cov\}\right)_{ijk\ell}
        &=-\frac{2\beta^2}{N}\left(
            \Cov_{ij}(\Cov^2)_{k\ell}
            +\Cov_{ik}(\Cov^2)_{j\ell}
            +\Cov_{i\ell}(\Cov^2)_{jk}
        \right)\mper
    \end{align}
    The triangle inequality and Cauchy--Schwarz imply
    \begin{align*}
        \|D^2U(m)\{\Cov,\Cov\}\|_F
        &\le\frac{6\beta^2}{N}\|m\|_2\|\Cov\|_F^2,\\
        \|D^2U(m)\{\kappa_3,\Cov\}\|_F
        &\le\frac{6\beta^2}{N}\|m\|_2\|\kappa_3\|_F\|\Cov\|_F,\\
        \|D^3U(m)\{\Cov,\Cov,\Cov\}\|_F
        &\le\frac{6\beta^2}{N}\|\Cov\|_F^3\mper
    \end{align*}
    By \cref{lem:cumulant-bounds},
    \[
        \E\|\Cov\|_F^p\le C_pN^{p/2}\mcom \qquad
        \E\|\kappa_3\|_F^p\le C_pN^{3p/4}\mper
    \]
    Combining these with $\|m\|_2\le\sqrt N$ and H\"older's inequality gives \cref{eq:D2-cov-cov,eq:D2-k3-cov,eq:D3-cov-cov-cov} with $U$ in place of $T$.
\end{proof}
\begin{corollary}\label{cor:r-grad12}
    For all $p\ge 2$, there exists $C_p > 0$ such that for all $|B|\le 2$, 
    \[
        \E \|\grad r^{(B)}\|_F^p
        \le C_p\mcom \qquad
        \E \|\grad^2 r^{(B)}\|_F^p
        \le C_p N^{p/4}\mper
    \]
\end{corollary}
\begin{proof}
    Combining the estimate on $\grad^j \lambda_i^{(B)}$ in \cref{lem:tap-derivative-estimates-new} (with $j=1,2$) with \cref{lem:improved-cumulant-kappa34} yields
    \[
        \E \|\grad \lambda_i^{(B)}\|_F^p
        \le C_p N^{-p/2} \mcom \qquad
        \E \|\grad^2 \lambda_i^{(B)}\|_F^p
        \le C_p N^{-p/4} \mper
    \]
    The result follows by differentiating $r_i^{(B)} = \zeta_i^{(B)} + m_i^{(B)} \lambda_i^{(B)}$ and applying the last display and the remaining estimates from \cref{lem:tap-derivative-estimates-new}.
\end{proof}

\begin{lemma}\label{lem:improved-cumulant-cov}
    For all $p\ge 2$, there exists $C_p > 0$ such that for all $|B| \le 2$,
    \[
        \E \|\Cov^{(B)}\|_{\mrop}^p
        \le C_p \mper
    \]
    In particular, \cref{lem:improved-cumulant}\ref{it:improved-cumulant-cov} follows as the case $|B| \le 1$.
\end{lemma}
\begin{proof} 
    We provide the proof for $B = \emptyset$ and omit the superscript $(\emptyset)$; the case of general $B$ is analogous.
    Let $\calE = \{K\succeq\gamma\Id\}$, where $\gamma>0$ is given by \cref{lem:organized-gap}.
    On $\calE$, the identity $K\Cov=\Id+\grad r$ from \cref{fac:improved-cumulant-start} gives
    \[
        \|\Cov\|_{\mrop}
        \le \|K^{-1}\|_{\mrop}\|\Id+\grad r\|_{\mrop}
        \le \gamma^{-1}\left(1+\|\grad r\|_F\right)\mper
    \]
    On the complement of $\calE$, we use the crude bound
    \[
        \|\Cov\|_{\mrop}\le\Tr\Cov=\sum_i(1-m_i^2)\le N\mper
    \]
    Thus,
    \begin{align*}
        \E\|\Cov\|_{\mrop}^p
        &\le \gamma^{-p}\E \left[\left(1 + \|\grad r\|_F\right)^p\right]
        +N^p\Pr[\calE^c] \\
        &\le 2^{p-1} \gamma^{-p} \left(1 + \E \|\grad r\|_F^p \right)
        +N^p\Pr[\calE^c]\mper
    \end{align*}
    By \cref{cor:r-grad12} and \cref{lem:organized-gap}, this is bounded by $C_p$.
\end{proof}

\begin{lemma}\label{lem:improved-cumulant-cov-high-prob}
    There exists a universal $C>0$ such that for all $a>0$, there exists $C_a > 0$ such that
    \[
        \Pr\left[
            \max_{|B| \le 2}
            \|\Cov^{(B)}\|_F
            > C\sqrt N
        \right]
        \le C_a N^{-a} \mper
    \]
    In particular, \cref{lem:improved-cumulant}\ref{it:improved-cumulant-cov-high-prob} follows by restricting the maximum to $|B| \le 1$.
\end{lemma}
\begin{proof}
    We will show that there exists a universal $C>0$ such that for all $a>0$, there exists $C_a > 0$ such that uniformly over $B$ with $|B| \le 2$,
    \begin{equation}\label{eq:improved-cumulant-cov-high-prob-goal}
        \Pr[\|\Cov^{(B)}\|_F > C\sqrt N] \le C_a N^{-a}\mper
    \end{equation} 
    The conclusion follows from \cref{eq:improved-cumulant-cov-high-prob-goal} by a union bound (and adjusting $a$).
    We provide the proof of \cref{eq:improved-cumulant-cov-high-prob-goal} for $B = \emptyset$ and omit the superscript $(\emptyset)$; the case of general $B$ is analogous.

    Let $\gamma>0$ be given by \cref{lem:organized-gap} and set $\calE = \{K\succeq\gamma\Id\}$.
    On $\calE$, \cref{fac:improved-cumulant-start} gives
    \[
        \|\Cov\|_F
        \le \|K^{-1}\|_{\mrop}\|\Id+\grad r\|_F
        \le \gamma^{-1}\left(\sqrt N+\|\grad r\|_F\right)\mper
    \]
    Take $p\ge\max(2,2a)$.
    By Markov's inequality, \cref{cor:r-grad12}, and \cref{lem:organized-gap},
    \begin{align*}
        \Pr\left[\|\Cov\|_F>2\gamma^{-1}\sqrt N\right]
        &\le \Pr[\calE^c]
        +\Pr\left[\|\grad r\|_F>\sqrt N\right] \\
        &\le \Pr[\calE^c]
        +N^{-p/2}\E\|\grad r\|_F^p
        \le C_a N^{-a}\mper
    \end{align*}
    This proves \cref{eq:improved-cumulant-cov-high-prob-goal} with $C=2\gamma^{-1}$.
\end{proof}

\begin{lemma}\label{lem:di-concentration}
    For all $p\ge 2$, there exists $C_p > 0$ such that for all $|B|\le 1$,
    \[
        \max_{i\in [N] \setminus B}
        \E \left|\|d_i^{(B)}\|_2^2 - \frac{\beta^2}{N} \|\Cov^{(B \cup \{i\})}\|_F^2 \right|^p
        \le C_p N^{-p/2}.
    \]
\end{lemma}
\begin{proof}
    We provide the proof for $B = \emptyset$ and omit the superscript $(\emptyset)$; the case of general $B$ is analogous.
    Fix $i\in[N]$ and set 
    \[
        \epsilon_i = d_i - \beta\Cov^{(i)}\bM_{i,-i}\mcom \qquad 
        H_i = (\Cov^{(i)})^2\mper
    \]
    Note that $v_i^{-1} \grad m_i = v_i^{-1} \Cov_i$ has $1$ in its $i$-th entry, and its remaining entries are $d_i$.
    Thus \cref{lem:tap-derivative-estimates-new} implies, for every $p\ge2$,
    \[
        \E\|d_i\|_2^p\le C_p\mcom\qquad
        \E\|\epsilon_i\|_2^p\le C_pN^{-p/2}\mper
    \]
    Note that
    \[
        \|d_i\|_2^2-\beta^2\bM_{i,-i}^\top H_i\bM_{i,-i}
        = \|d_i\|_2^2-\|d_i-\epsilon_i\|_2^2
        = 2\langle d_i,\epsilon_i\rangle - \|\epsilon_i\|_2^2\mper
    \]
    By H\"older's inequality,
    \begin{equation}\label{lem:di-concentration-step1}
        \E\left|\|d_i\|_2^2-\beta^2\bM_{i,-i}^\top H_i\bM_{i,-i}\right|^p
        \le C_pN^{-p/2}\mper
    \end{equation}
    Conditional on $(\bx,\bW_{-i,-i},\bg)$, the matrix $H_i$ is fixed and
    \[
        \bM_{i,-i}=\bW_{i,-i}+\eta_i\mcom\qquad
        \eta_i=\frac{\beta\bx_i\bx_{-i}}{N}\mcom
    \]
    where $\bW_{i,-i}\sim\calN(0,\Id/N)$ and $\|\eta_i\|_2\le\frac{\beta}{\sqrt N}$.
    Writing $\E_i$ for expectation over $\bW_{i,-i}$ under this conditioning, standard Gaussian quadratic- and linear-form moment bounds give
    \[
        \E_i\left|\bM_{i,-i}^\top H_i\bM_{i,-i}-\frac{\Tr H_i}{N}\right|^p
        \le \frac{C_p}{N^p} \|H_i\|_F^p
        +\frac{\|H_i\eta_i\|_2^p}{N^{p/2}}
        +|\eta_i^\top H_i\eta_i|^p
        \le \frac{C_p}{N^p} \|H_i\|_F^p\mper
    \]
    Since $\|H_i\|_F\le\sqrt N\|\Cov^{(i)}\|_{\mrop}^2$, \cref{lem:improved-cumulant-cov} implies
    \[
        \E\left|\bM_{i,-i}^\top H_i\bM_{i,-i}-\frac{\Tr H_i}{N}\right|^p
        \le C_pN^{-p/2}\mper
    \]
    Combining with \cref{lem:di-concentration-step1} and $\Tr H_i=\|\Cov^{(i)}\|_F^2$ proves the result.
\end{proof}

\begin{proof}[Proof of \cref{lem:improved-cumulant}\ref{it:improved-cumulant-kappa3-high-prob}]
    We will show that there exists a universal $C>0$ such that for all $a>0$, there exists $C_a>0$ such that uniformly over $B$ with $|B|\le 1$,
    \begin{equation}\label{eq:improved-cumulant-kappa3-high-prob-goal}
        \Pr[\|\kappa_3^{(B)}\|_F>C\sqrt N]\le C_aN^{-a}\mper
    \end{equation}
    The conclusion follows by a union bound (and adjusting $a$).
    We provide the proof of \cref{eq:improved-cumulant-kappa3-high-prob-goal} for $B = \emptyset$ and omit the superscript $(\emptyset)$; the case of general $B$ is analogous.

    Let $A,U$ be defined as in \eqref{eq:AU}; as observed in the proof of \cref{lem:improved-cumulant-kappa34}, $T - A - U$ is linear.
    By \cref{fac:improved-cumulant-start},
    \begin{equation}\label{eq:improved-cumulant-kappa3-high-prob-start}
        K\kappa_3 = \grad^2 r - D^2 A(m)\{\Cov,\Cov\} - D^2 U(m)\{\Cov,\Cov\}\mper
    \end{equation}
    By \cref{eq:D2U} and the estimates $\|m\|_2 \le \sqrt{N}$, $\|\Cov\|_F \le \sqrt{N} \|\Cov\|_{\mrop}$,
    \[
        \|D^2U(m)\{\Cov,\Cov\}\|_F
        \le \frac{6\beta^2}{N}\|m\|_2\|\Cov\|_{\mrop}\|\Cov\|_F
        \le 6\beta^2\|\Cov\|_{\mrop}^2\mper
    \]
    Together with \cref{cor:r-grad12,lem:improved-cumulant-cov}, this implies, for every fixed $p\ge 2$,
    \[
        \max\left\{
            \E\|\grad^2 r\|_F^p,
            \E\|D^2U(m)\{\Cov,\Cov\}\|_F^p
        \right\} \le C_pN^{p/4}\mper
    \]
    By Markov's inequality with $p\ge\max(2,4a)$,
    \begin{equation}\label{eq:grad2r-D2U-high-prob-ub}
        \Pr\left(
            \max(\|\grad^2 r\|_F, \|D^2U(m)\{\Cov,\Cov\}\|_F)
            \ge \sqrt{N}
        \right) \le C_a N^{-a}\mper
    \end{equation}
    We turn to controlling the term $\|D^2 A(m)\{\Cov,\Cov\}\|_F$.
    As noted in the proof of \cref{lem:di-concentration}, the $i$-th coordinate of $v_i^{-1}\grad m_i$ is $1$, and the remaining coordinates are $d_i$.
    So,
    \[
        \|D^2A(m)\{\Cov,\Cov\}\|_F^2
        = 4\sum_i m_i^2\|v_i^{-1}\grad m_i\|_2^4
        \le 4\sum_i(1+\|d_i\|_2^2)^2\mper
    \]
    Set $q_i = \beta^2 \|\Cov^{(i)}\|_F^2 / N$ and
    \[
        \Gamma = \sum_i \left|\|d_i\|_2^2 - q_i \right|^2\mper
    \]
    By Jensen's inequality and \cref{lem:di-concentration}, for every $p\ge 1$,
    \[
        \E \Gamma^p
        \le N^{p-1} \sum_i \E \left|\|d_i\|_2^2 - q_i \right|^{2p}
        \le C_p\mper
    \]
    By Markov's inequality with $p \ge \max(1,a)$, $\Pr(\Gamma \le N) \ge 1 - C_a N^{-a}$.
    \cref{lem:improved-cumulant-cov-high-prob} provides a universal constant $C_0$ such that
    \[
        \Pr(\max_i q_i \le \beta^2 C_0^2)
        \ge 1 - C_a N^{-a}\mper
    \]
    On the intersection of these two events,
    \[
        \|D^2A(m)\{\Cov,\Cov\}\|_F^2
        \le 4 \sum_i(1+\|d_i\|_2^2)^2
        \le 8\sum_i (1 + q_i)^2 + 8\Gamma
        \le 8N\left((1 + \beta^2 C_0^2)^2 + 1\right)\mper
    \]
    This provides an absolute $C$ such that
    \[
        \Pr(\|D^2A(m)\{\Cov,\Cov\}\|_F > C\sqrt{N}) \le C_a N^{-a}\mper
    \]
    Recalling \eqref{eq:improved-cumulant-kappa3-high-prob-start}, a union bound with \cref{eq:grad2r-D2U-high-prob-ub} shows
    \[
        \Pr(\|K\kappa_3\|_F > C \sqrt{N}) \le C_a N^{-a}\mcom
    \]
    after adjusting $C$ and $C_a$.
    Combining with \cref{lem:organized-gap} proves the result.
\end{proof}

We finally present the deferred proof of \cref{lem:tap-derivative-estimates-new}.
\begin{proof}[Proof of \cref{lem:tap-derivative-estimates-new}]
    We provide the proof for $B = \emptyset$ and omit the superscript $(\emptyset)$; the case of general $B$ is analogous.
    The proof proceeds in five steps, and uses only \cref{lem:cumulant-bounds,lem:cumulant-contraction-estimates} and Gaussian moment estimates.

    \textit{Step 1: cavity expansion for $d_i$.}
    Let $\epsilon_i = d_i - \beta \Cov^{(i)} \bM_{i,-i}$. In this step we prove
    \begin{equation}\label{eq:tap-derivative-estimates-new-step1}
        \E\|\epsilon_i\|_2^p\le C_pN^{-p/2}\mcom
    \end{equation}
    which is the first estimate of the lemma.
    For $s' \in \{\pm 1\}$, Taylor expansion gives
    \[
        m^{(i,s')}
        = \nabla F^{(i)}(u_i + s'\beta \bM_{i,-i})
        = \nabla F^{(i)}(u_i)
        + s' \beta \Cov^{(i)} \bM_{i,-i}
        + \frac{\beta^2}{2} \kappa_3^{(i)} [\bM_{i,-i}^{\otimes 2}]
        + \epsilon_{i,s'},
    \]
    where
    \[
        \epsilon_{i,s'} = \frac{s'\beta^3}{2}\int_0^1(1-s)^2
        \kappa_{4,ss'}^{(i)} [\bM_{i,-i}^{\otimes 3}]\,\dif s\mper
    \]
    Thus
    \[
        d_i = \frac{m^{(i,+)}-m^{(i,-)}}{2}
        = \beta \Cov^{(i)} \bM_{i,-i} + \frac{\epsilon_{i,+} - \epsilon_{i,-}}{2}\mcom
    \]
    and
    \begin{equation}\label{eq:tap-d-taylor-remainder}
        \epsilon_i = \frac{\epsilon_{i,+} - \epsilon_{i,-}}{2}
        =\frac{\beta^3}{4}\int_0^1(1-s)^2
        \left(\kappa_{4,s}^{(i)}+\kappa_{4,-s}^{(i)}\right)
        [\bM_{i,-i}^{\otimes3}]\,\dif s\mper
    \end{equation}
    Applying \cref{lem:cumulant-contraction-estimates} with $r=4$, $k=0$, and $\ell=3$ yields \cref{eq:tap-derivative-estimates-new-step1}.
    
    \textit{Step 2: moments of $d_i$ and its derivatives.}
    In this step, we prove the auxiliary estimate
    \begin{equation}\label{eq:tap-d-derivatives-aux}
        \E\|\grad^j d_i\|_F^p\le C_pN^{pj/4}
    \end{equation}
    for $j=0,1,2,3$.
    Conditional on $(\bx,\bW_{-i,-i},\bg)$, the matrix $\Cov^{(i)}$ is fixed and
    \[
        \bM_{i,-i}=\bW_{i,-i}+\eta_i\mcom\qquad
        \eta_i=\frac{\beta\bx_i\bx_{-i}}{N}\mcom
    \]
    where $\bW_{i,-i} \sim \calN(0, \Id/N)$ and $\|\eta_i\|_2 \le \beta / \sqrt{N}$.
    Write $\E_i$ for expectation over $\bW_{i,-i}$ under this conditioning.
    Gaussian moments give
    \[
        \E_i\|\Cov^{(i)}\bM_{i,-i}\|_2^p
        \le \frac{C_p}{N^{p/2}}\|\Cov^{(i)}\|_F^p
        + \|\Cov^{(i)}\eta_i\|_2^p
        \le \frac{C_p}{N^{p/2}}\|\Cov^{(i)}\|_F^p\mper
    \]
    Taking expectation and using the covariance bound from \cref{lem:cumulant-bounds}, we obtain
    \[
        \E\|\Cov^{(i)}\bM_{i,-i}\|_2^p\le C_p
    \]
    Together with \cref{eq:tap-derivative-estimates-new-step1} this implies \cref{eq:tap-d-derivatives-aux} for $j=0$.
    For $j=1,2,3$, differentiating $d_i=(m^{(i,+)}-m^{(i,-)})/2$ and then integrating in $s$ gives
    \[
        \grad^j d_i
        =\frac{\kappa_{j+1,1}^{(i)}-\kappa_{j+1,-1}^{(i)}}2
        =\frac\beta2\int_{-1}^1
        \kappa_{j+2,s}^{(i)}[\bM_{i,-i}]\,\dif s\mper
    \]
    Applying \cref{lem:cumulant-contraction-estimates} with $r=j+2$, $k=0$, and $\ell=1$ proves \cref{eq:tap-d-derivatives-aux}.

    \textit{Step 3: derivatives of $m_i$ and $\zeta_i$.}
    In this step we prove the lemma's estimates on $(v_i)^{-1} \grad^j m_i$ and $\grad^j \zeta_i$.
    Set $a_i=\arctanh(m_i)$.
    By the definition of $d_i$,
    \[
        \grad a_i=v_i^{-1}\grad m_i=v_i^{-1}\Cov_i=(1,d_i),
    \]
    where we write the $i$-th coordinate first.
    Differentiating this identity and using \cref{eq:tap-d-derivatives-aux} gives
    \begin{equation}\label{eq:tap-a-derivatives-aux}
        \E\|\grad^j a_i\|_F^p\le C_pN^{p(j-1)/4}
    \end{equation}
    for $j=1,2,3$.
    Since $m_i=\tanh a_i$ and $v_i=1-m_i^2$, we calculate
    \[
        \frac{\tanh'(a_i)}{v_i}=1\mcom\qquad
        \frac{\tanh''(a_i)}{v_i}=-2m_i\mcom\qquad
        \frac{\tanh'''(a_i)}{v_i}=6m_i^2-2\mper
    \]
    Thus the chain rule gives
    \begin{align*}
        v_i^{-1}\grad m_i
        &=\grad a_i\mcom\\
        v_i^{-1}\grad^2m_i
        &=\grad^2a_i-2m_i(\grad a_i)^{\otimes2}\mcom\\
        v_i^{-1}\grad^3m_i
        &=\grad^3a_i-2m_i\sum_{\Sym}\grad^2a_i\otimes\grad a_i
        +(6m_i^2-2)(\grad a_i)^{\otimes3}\mper
    \end{align*}
    Since $|m_i|\le1$, all coefficients are bounded.
    Applying \cref{eq:tap-a-derivatives-aux} proves the desired estimates on $v_i^{-1}\grad^j m_i$.
    Finally, differentiating the integral representation in \cref{eq:zeta} yields
    \[
        \grad^j\zeta_i
        =-\frac{\beta^3}{4}\int_{-1}^1(1-s^2)
        \kappa_{j+3,s}^{(i)}[\bM_{i,-i}^{\otimes3}]\,\dif s\mper
    \]
    Applying \cref{lem:cumulant-contraction-estimates} with $r=j+3$, $k=0$, and $\ell=3$ proves the claimed estimates on $\grad^j\zeta_i$.
    
    \textit{Step 4: cavity estimate for covariance trace.}
    Set $\Xi_i=\Tr\Cov-\Tr\Cov^{(i)}$.
    In this step, we prove the auxiliary estimate
    \begin{equation}\label{eq:tap-cavity-trace-derivatives}
        \E\|\grad^j\Xi_i\|_F^p\le C_pN^{p(j+1)/4}
    \end{equation}
    for $j=0,1,2,3$.
    Note that here, $\grad$ includes derivatives in coordinate $i$ of the external field $u = h$, even though $\Cov^{(i)}$ does not depend on this coordinate.

    Let $\bsigma$ be a sample from the Gibbs measure.
    The law of total covariance yields
    \begin{align*}
        \Xi_i
        &= \Tr \E[\Cov(\bsigma | \sigma_i)] 
        - \Tr \Cov^{(i)} 
        + \Tr \Cov(\E[\bsigma | \sigma_i]) \\
        &= \frac{1+m_i}{2} \kappa_{2,1}^{(i)}[\Id]
        + \frac{1-m_i}{2} \kappa_{2,-1}^{(i)}[\Id]
        - \kappa_2^{(i)}[\Id]
        + v_i(1+\|d_i\|_2^2)
        = \frac12 \Xi_{i,1} + \frac{m_i}{2} \Xi_{i,2} 
        + v_i(1+\|d_i\|_2^2)\mcom
    \end{align*}
    where
    \[
        \Xi_{i,1} = \left(\kappa_{2,1}^{(i)} + \kappa_{2,-1}^{(i)} - 2 \kappa_2^{(i)}\right)[\Id]\mcom\qquad
        \Xi_{i,2} = \left(\kappa_{2,1}^{(i)} - \kappa_{2,-1}^{(i)}\right)[\Id]\mper
    \]
    For $j=0,1,2,3$, Taylor expansion in $s$ gives
    \begin{align*}
        \grad^j \Xi_{i,1}
        = \left(
            \kappa_{j+2,1}^{(i)}
            + \kappa_{j+2,-1}^{(i)}
            - 2\kappa_{j+2}^{(i)}
        \right)[\Id]
        &=\beta^2\int_0^1(1-s)
        \left(
            \kappa_{j+4,s}^{(i)}
            +\kappa_{j+4,-s}^{(i)}
        \right)
        [\Id\otimes\bM_{i,-i}^{\otimes2}]\,\dif s,\\
        \grad^j \Xi_{i,2} = \left(
            \kappa_{j+2,1}^{(i)}
            -\kappa_{j+2,-1}^{(i)}
        \right)[\Id]
        &=\beta\int_{-1}^1
        \kappa_{j+3,s}^{(i)}[\Id\otimes\bM_{i,-i}]\,\dif s\mper
    \end{align*}
    Applying \cref{lem:cumulant-contraction-estimates} with $(r,k,\ell)=(j+4,1,2)$ and $(j+3,1,1)$ respectively, and integrating in $s$, bounds
    \[
        \E \|\grad^j \Xi_{i,1}\|_F^p \le C_p N^{pj/4}\mcom\qquad
        \E \|\grad^j \Xi_{i,2}\|_F^p \le C_p N^{p(j+1)/4}\mper
    \]
    From the conclusion of step 3 and the trivial bound $v_i \le 1$, we have $\E\|\grad^j m_i\|_F^p\le C_pN^{p(j-1)/4}$ for $j=1,2,3$.
    Differentiating $v_i=1-m_i^2$ gives the same bound for $\grad^j v_i$; thus
    \begin{equation}\label{eq:tap-m-v-derivatives-aux}
        \max\left\{
            \E\|\grad^j m_i\|_F^p\mcom
            \E\|\grad^j v_i\|_F^p
        \right\}
        \le C_pN^{p(j-1)/4}
    \end{equation}
    for $j=1,2,3$.
    Combining \cref{eq:tap-d-derivatives-aux,eq:tap-m-v-derivatives-aux} with $0\le v_i\le1$ yields
    \[
        \E\left\|\grad^j\left(v_i(1+\|d_i\|_2^2)\right)\right\|_F^p
        \le C_pN^{pj/4}
    \]
    for $j=0,1,2,3$.
    Together this proves \cref{eq:tap-cavity-trace-derivatives}.

    \textit{Step 5: derivatives of $\lambda_i$.}
    In this step we prove the lemma's estimates on $\grad^j\lambda_i$.
    Using $\Tr\Cov=N-\|m\|_2^2$ and the definitions $\epsilon_i = d_i - \beta\Cov^{(i)}\bM_{i,-i}$ from Step 1 and $\Xi_i = \Tr \Cov - \Tr \Cov^{(i)}$ from Step 4, we obtain
    \[
        \lambda_i
        =\frac{\beta^2}{N}\Tr\Cov-\beta\bM_{i,-i}^\top d_i
        =\frac{\beta^2}{N}\Xi_i
        -\beta^2\Cov^{(i)}\left[\bM_{i,-i}^{\otimes2}-\frac{\Id}{N}\right]
        -\beta\bM_{i,-i}^\top\epsilon_i\mper
    \]
    Differentiating with respect to the external field gives
    \begin{equation}\label{eq:tap-centered-lambda-derivatives}
        \grad^j\lambda_i
        =\frac{\beta^2}{N}\grad^j\Xi_i
        -\beta^2\kappa_{j+2}^{(i)}
        \left[\bM_{i,-i}^{\otimes2}-\frac{\Id}{N}\right]
        -\beta\grad^j(\bM_{i,-i}^\top\epsilon_i)\mcom
    \end{equation}
    for $j=0,1,2,3$.
    Since $\bM_{i,-i}$ does not depend on the external field, differentiating \cref{eq:tap-d-taylor-remainder} gives
    \[
        \grad^j(\bM_{i,-i}^\top\epsilon_i)
        = \frac{\beta^3}{4}
        \int_0^1(1-s)^2
        \left(\kappa_{j+4,s}^{(i)}+\kappa_{j+4,-s}^{(i)}\right)
        [\bM_{i,-i}^{\otimes4}]\,\dif s\mper
    \]
    By \cref{lem:cumulant-contraction-estimates} with $r=j+4$, $k=0$, and $\ell=4$,
    \begin{equation}\label{eq:grad-Meps}
        \E\|\grad^j(\bM_{i,-i}^\top\epsilon_i)\|_F^p\le C_pN^{p(j-4)/4}\mper
    \end{equation}
    We will bound the $\grad^j \Xi_i$ term in \cref{eq:tap-centered-lambda-derivatives} using \cref{eq:tap-cavity-trace-derivatives}.
    We now turn to the $\kappa_{j+2}^{(i)}$ term.
    Recall from Step 2 the notation $\bM_{i,-i}=\bW_{i,-i}+\eta_i$, and that $\E_i$ denotes expectation over $\bW_{i,-i}$ conditional on $(\bx,\bW_{-i,-i},\bg)$.
    Then
    \[
        \bM_{i,-i}^{\otimes2}-\frac{\Id}{N}
        =\left(\bW_{i,-i}^{\otimes2}-\frac{\Id}{N}\right)
        +\bW_{i,-i}\otimes\eta_i+\eta_i\otimes\bW_{i,-i}
        +\eta_i^{\otimes2}\mper
    \]
    For any tensor $H$ of order at least two measurable with respect to $(\bx,\bW_{-i,-i},\bg)$, Gaussian quadratic- and linear-form moment bounds give
    \begin{align*}
        \E_i\left\|H\left[\bM_{i,-i}^{\otimes2}-\frac{\Id}{N}\right]\right\|_F^p
        \le C_p\left(\frac1N+\frac{\|\eta_i\|_2}{\sqrt N}+\|\eta_i\|_2^2\right)^p \|H\|_F^p
        \le\frac{C_p}{N^p}\|H\|_F^p\mper
    \end{align*}
    The last inequality uses $\|\eta_i\|_2\le\beta/\sqrt N$ from Step 2.
    Taking $H=\kappa_{j+2}^{(i)}$, which is fixed under the conditioning, and taking expectation yields
    \[
        \E\left\|\kappa_{j+2}^{(i)}
        \left[\bM_{i,-i}^{\otimes2}-\frac{\Id}{N}\right]\right\|_F^p
        \le\frac{C_p}{N^p}\E\|\kappa_{j+2}^{(i)}\|_F^p\mper
    \]
    Combining this with \cref{eq:tap-cavity-trace-derivatives,eq:tap-centered-lambda-derivatives,eq:grad-Meps} proves
    \[
        \E\|\grad^j\lambda_i\|_F^p
        \le C_p\left(N^{p(j-3)/4}+N^{-p}\E\|\kappa_{j+2}^{(i)}\|_F^p\right)
    \]
    for $j=0,1,2,3$, as desired.
    If $j=0$, \cref{lem:cumulant-bounds} gives $\E\|\kappa_2^{(i)}\|_F^p\le C_pN^{p/2}$, and we obtain the final assertion $\E|\lambda_i|^p \le C_pN^{-p/2}$ of the lemma.
\end{proof}

\end{document}